\documentclass[12pt,reqno]{amsart}
\usepackage{latexsym,amsmath,amsfonts,amscd,amssymb}
\usepackage{graphics}

\usepackage{float}
\usepackage{tikz}
\usepackage{dsfont}
\usetikzlibrary{matrix}

\usepackage{tikz-cd}

\newcommand{\la}{\langle}
\newcommand{\ra}{\rangle}
\newcommand{\x}{\times}
\newcommand{\ox}{\otimes}
\newcommand{\too}{\to}
\newcommand{\bd}{\partial}
\newcommand{\cA}{{\mathcal{A}}}

\newcommand{\cF}{{\mathcal{F}}}
\newcommand{\cU}{{\mathcal{U}}}
\newcommand{\cV}{{\mathcal{V}}}
\newcommand{\ZZ}{\mathbb{Z}}
\newcommand{\CC}{\mathbb{C}}
\newcommand{\CP}{\mathbb{C}P}
\newcommand{\RP}{\mathbb{R}P}
\newcommand{\RR}{\mathbb{R}}

\newcommand{\QQ}{\mathbb{Q}}
\newcommand{\id}{\operatorname{id}}
\newcommand{\im}{\operatorname{im}}

\newcommand{\U}{\operatorname{U}}
\newcommand{\SO}{\operatorname{SO}}
\newcommand{\GL}{\operatorname{GL}}

\DeclareMathOperator{\lcm}{lcm}
\DeclareMathOperator{\coker}{coker}
\DeclareMathOperator{\Id}{Id}

\DeclareMathOperator{\End}{End}

\renewcommand{\a}{\alpha}
\renewcommand{\b}{\beta}

\newcommand{\g}{\gamma}
\newcommand{\e}{\varepsilon}

\newcommand{\s}{\sigma}

\renewcommand{\o}{\omega}

\newcommand{\G}{\Gamma}

\renewcommand{\S}{\Sigma}

\newcommand{\OO}{\mathrm{O}}
\newcommand{\UU}{\mathrm{U}}
\newcommand{\pr}{\mathrm{pr}}

\newcommand{\orb}{\scriptsize\mathrm{orb}}

\newtheorem{theorem}{Theorem}
\newtheorem{proposition}[theorem]{Proposition}
\newtheorem{lemma}[theorem]{Lemma}
\newtheorem{definition}[theorem]{Definition}

\newtheorem{corollary}[theorem]{Corollary}
\newtheorem{remark}[theorem]{Remark}

\title[Orbifold Gompf connected sum and K-contact manifolds]{Gompf connected sum for orbifolds and K-contact Smale-Barden manifolds}

\author[V. Mu\~{n}oz]{Vicente Mu\~{n}oz}
\address{Departamento de \'Algebra, Geometr\'{\i}a y Topolog\'{\i}a, Universidad de M\'alaga, 
Campus de Teatinos, s/n, 29071 M\'alaga, Spain}
\email{vicente.munoz@uma.es}



\begin{document}

\begin{abstract}
 We develop the Gompf fiber connected sum operation for symplectic orbifolds. We use it
 to construct a symplectic $4$-orbifold with $b_1=0$ and containing symplectic surfaces
 of genus $1$ and $2$ that are disjoint and span the rational homology. This is used
 in turn to construct a K-contact Smale-Barden manifold with specified $2$-homology that 
 satisfies the known topological constraints with sharper estimates than the examples 
 constructed previously. The manifold can be chosen spin or non-spin. 
\end{abstract}

\maketitle

\section{Introduction}\label{sec:intro}

In geometry, a central question is to determine when a given manifold admits a specific geometric structure. 
Complex geometry provides with numerous examples of compact manifolds with rich topology, and there is a number
of topological properties that are satisfied by K\"ahler manifolds. 
If we forget about the integrability of the complex structure, then we are dealing with symplectic manifolds. There has
been enormous interest in the construction of (compact) symplectic manifolds that do not admit K\"ahler structures,
and in determining its topological properties \cite{OT}. 
In odd dimension, Sasakian and K-contact manifolds are natural analogues of K\"ahler and symplectic manifolds, respectively. 
The precise definition of such structures is recalled in section \ref{sec:k-contact}.
Sasakian geometry has become an important and active subject since \cite{BG}, 
and there is much interest on constructing K-contact manifolds which do not admit Sasakian structures.

The problem of the existence of simply connected K-contact non-Sasakian compact manifolds  (open problem 7.4.1 in \cite{BG})
was solved for dimensions $\geq 9$  
in \cite{CNY,CNMY,HT} and for dimension $7$ in \cite{MT} by a combination of various techniques 
in homotopy theory and symplectic geometry, but it is still open in dimension $5$ (cf.\ \cite[open problem 10.2.1]{BG}).
A simply connected compact $5$-manifold is called a {\it Smale-Barden manifold}. 
These manifolds are classified \cite{B,S} by their second homology group, that we write as 
  \begin{equation}\label{eqn:H2-1}
  H_2(M,\ZZ)=\ZZ^k\oplus( \mathop{\oplus}_{p,i}\, \ZZ_{p^i}^{c(p^i)}),
  \end{equation}
where $k=b_2(M)$, and its second Stiefel-Whitney class $w_2$, which is 
zero on all but one summand $\ZZ_{2^j}$, where the value $j=i(M)$ is the Barden invariant.

A Sasakian (compact) manifold $M$ has a $1$-dimensional foliation defined by the Reeb vector field, which gives an isometric flow,
and the transversal structure is K\"ahler. The Sasakian structure is called \emph{quasi-regular} if the leaves of the Reeb flow are
circles, in which case the leaf space $X$ is a cyclic K\"ahler orbifold and
the quotient map $\pi:M\to X$ has the structure of a Seifert circle bundle. Remarkably, a manifold $M$ admitting
a Sasakian structure also has a quasi-regular one \cite{R}. So from the point of view of whether $M$ admits a Sasakian 
structure, we can assume that it is a Seifert circle bundle over a cyclic K\"ahler orbifold. The Sasakian structure
is \emph{regular} if $X$ is a K\"ahler manifold (no isotropy locus), and \emph{semi-regular} if the isotropy locus
has only codimension $2$ strata (maybe intersecting), or equivalently if $X$ has underlying space which is a topological manifold.
In \cite{K} Koll\'ar studies the topology of semi-regular Seifert bundles $M\to X$ when $H_1(M,\ZZ)=0$. Under some
technical conditions, we have 
  \begin{equation}\label{eqn:H2-2}
  H_2(M,\ZZ)=\ZZ^k\oplus (\mathop{\oplus}_i \, \ZZ_{m_i}^{2g_i}),
  \end{equation}
where $H_1(X,\ZZ)=0$, $H_2(X,\ZZ)=\ZZ^{k+1}$, the isotropy locus are complex curves $D_i$ of genus $g(D_i)= g_i$,
with isotropy coefficients $m_i$ such that $\gcd(m_i,m_j)=1$ when $D_i,D_j$ intersect, and $[D_i]$ are linearly independent in homology when the 
coefficients are not coprime. This is used in \cite{K} to obtain Seifert circle bundles which cannot be Sasakian for $k=0$.

In the case of a K-contact manifold, the situation is analogous, with the difference that 
the transversal structure is almost-K\"ahler.
We define regular, quasi-regular and semi-regular K-contact structures
with the same conditions. Any K-contact manifold admits a quasi-regular K-contact structure  \cite{MT}, and
hence a K-contact manifold is a Seifert circle bundle over a cyclic symplectic orbifold. Such orbifold has isotropy locus
which are a (stratified) collection of symplectic suborbifolds. 
The homology (\ref{eqn:H2-2}) tells us interesting geometric facts: we can read the genus of the isotropy surfaces $D_i$
as long as $g_i>0$, and also that they are disjoint when $m_i$ are not coprime but different. As $[D_i]$ are
linearly independent in homology, we cannot have too many disjoint surfaces, and the hardest possible situation
is when there are $k+1=b_2(X)$ disjoint surfaces $D_i$ (taking all $m_i$ not pairwise coprime).
This is used in \cite{MRT} to construct a K-contact $5$-manifold $M$ which cannot be Sasakian, specifically for
$k+1=36$, with $H_1(M,\ZZ)=0$ and semi-regular structure. The distinctive geometric property is that there are 
symplectic $4$-manifolds $X$ with $b_2(X)$ disjoint symplectic surfaces of positive genus (and independent in homology), 
whereas this seems to be difficult (conjecturally impossible) for algebraic surfaces and complex curves (with the
exception of fake projective planes, which have $b_2=1$). 

In \cite{CMRV} there is a second construction of a symplectic $4$-manifold with $b_1=0$ and $b_2=k+1=12$, and
with $12$ disjoint symplectic surfaces of positive genus and independent in homology. In this case the 
corresponding $5$-manifold $M$ has trivial fundamental group, and therefore it is a Smale-Barden manifold
which admits a semi-regular K-contact structure but not a semi-regular Sasakian structure. Both
examples \cite{CMRV,MRT} have surfaces of small genus $g_i\in \{1,2,3\}$, which are cases where we can control
the impossibility of having that many disjoint complex curves in an algebraic surface. Noticeably, 
there is always at least a surface of genus $3$, and it seems difficult to lower the genus of the surfaces.

Another consequence of (\ref{eqn:H2-2}) is that Seifert circle bundles $M\to X$ over cyclic orbifolds satisfy the 
{\it G-K condition}, which means that, in terms of the alternative expression (\ref{eqn:H2-1}):
 \begin{itemize}
 \item for every prime $p$, $t(p)=\#\{i\, | \, c(p^i)>0\}\leq k+1$,
 \item  $i(M)\in\{0,\infty\}$; if $i(M)=\infty$ ($M$ non-spin), then $t(2) \leq k$.
 \end{itemize}
The calculation of the second Stiefel-Whitney class appears in \cite{K,MT2}.
In \cite[Question 10.2.1]{BG} it is asked whether a Smale-Barden manifold which 
satisfies the G-K conditon admits a Sasakian structure.
Write
 $$
 \mathbf{t}(M)=\max\{ t(p)|\,p \text{  prime}\} \leq k+1.
 $$
The difficulty to obtain examples increase as we go to the upper bound, since we always can discard surfaces from the
isotropy locus. The examples of \cite{CMRV,MRT} are instances where the upper
bound $\mathbf{t}(M)=k+1$ is achieved.

Note that the case $\mathbf{t} =0$ is that of torsion-free Smale-Barden manifolds, where
we only have regular Sasakian structures, and all G-K manifolds admit Sasakian structures.
The next case is $\mathbf{t}=1$, which is studied in detail in \cite{MT2}. 
All G-K manifolds with $\mathbf{t}=1$ and $k\geq 1$ admit semi-regular Sasakian structures, and
hence the manifolds admitting Sasakian and K-contact structures are the same.
In the borderline case $\mathbf{t}=1,k=0$, the results in \cite{MT2} are only partial
and touch open questions on symplectic $4$-manifold topology.

Write also
 $$
 \mathbf{c}(M)=\max\{ c(p^i)\} =\max \{g_i\}.
 $$
Our previous comments indicate that it is hard to get examples with low $\mathbf{c}(M)$ and $\mathbf{t}(M)=k+1$, $k=b_2(M)$.
One of the purposes of the present paper is to show the following.

\begin{theorem} \label{thm:main}
 There is a simply connected $5$-manifold $M$ admitting a (quasi-regular) K-contact structure with 
 $\mathbf{t}(M)=b_2(M)+1$, $\mathbf{c}(M)=2$. Such $M$ can be chosen spin or non-spin.
\end{theorem}
 
The construction of Theorem \ref{thm:main} comes down to the construction of a suitable symplectic cyclic $4$-orbifold
with many disjoint symplectic surfaces.
 
\begin{theorem} \label{thm:main2}
 There is a symplectic cyclic $4$-orbifold $X$ with $b_1=0$, $b_2=16$, thirteen symplectic surfaces of genus $1$ and 
 three of genus $2$, which are disjoint and generate the homology $H_2(X,\QQ)$. 
 \end{theorem}

To construct the orbifold of Theorem \ref{thm:main2}, we will develop in Sections \ref{sec:orbifolds}-\ref{sec:3}
the technique introduced by Gompf \cite{Gompf}
for fiber connected sum of symplectic manifolds along codimension $2$ symplectic submanifolds, in the orbifold setting.
This technique will be useful by its own in future constructions of symplectic orbifolds.
In Sections \ref{sec:4}-\ref{sec:5} we do the construction of the symplectic $4$-orbifold of Theorem \ref{thm:main2}.
Then we have to extend the theory of \cite{K,MRT} for semi-regular Seifert bundles to the case of quasi-regular Seifert bundles,
in Section \ref{sec:k-contact}. In Section \ref{sec:w2} we compute the second Stiefel-Whitney class extending the
arguments of \cite{MT} to the quasi-regular setting. Finally, we compute the orbifold fundamental group of our
symplectic $4$-orbifold in Section \ref{sec:pi1}.

\subsection*{Acknowledgements} 
The author is grateful to Jaume Amor\'os, Denis Auroux, Alejandro Ca\~nas, Paul Seidel and Alex Tralle for very useful comments.
Partially supported by Project MINECO (Spain) PGC2018-095448-B-I00.

\section{Orbifolds}\label{sec:orbifolds}

\subsection{Orbifolds}
Let us start by  collecting  some results about orbifolds from \cite{BFMT, BG}.
Let $X$ be a topological space, and fix an integer $n > 0$. 
An {\em orbifold chart} $(U, {\tilde U}, \Gamma, \varphi)$ at $x\in X$ 
consists of an open set $U  \subset  X$ with $x\in U$, a connected and open set
${\tilde U}  \subset  \RR^n$, a finite group $\Gamma\subset \GL(n)$ acting smoothly and effectively 
on $\tilde{U}$ fixing $0$, and a continuous map $\varphi \colon \tilde{U} \to U$, with $\varphi(0)=x$, which is  $\Gamma$-invariant 
(that is $\varphi = \varphi\circ\gamma$, for all $\gamma \in\Gamma$), that induces a homeomorphism 
$\tilde U/\G  \stackrel{\cong}{\too} U$.

\begin{definition}\label{def:orbifold}
An orbifold $X$, of dimension $n$, is a Hausdorff, paracompact topological 
 space endowed with an {\em orbifold atlas} $\cA=\{(U_i, {\tilde U}_i, \Gamma_i, \varphi_i)\}$ of orbifold charts
which satisfy the following conditions:
 \begin{enumerate}
 \item[i)] $\{U_i\}$ is an open cover of $X$;
 \item[ii)] If $(U_i, {\tilde U}_i, \Gamma_i, \varphi_i)$ and $(U_j, {\tilde U}_j, \Gamma_j, \varphi_j)$
 are two orbifold charts, with $U_i \cap U_j \not=\emptyset$, then for each point $p \in U_i \cap U_j$ there exists
an orbifold chart $(U_k, {\tilde U}_k, \Gamma_k, \varphi_k)$ at $p$ such that $U_k \subset U_i \cap U_j$;
 \item[iii)] If $(U_i, {\tilde U}_i, \Gamma_i, \varphi_i)$ and $(U_j, {\tilde U}_j, \Gamma_j, \varphi_j)$
 are two orbifold charts at $p\in X$, with $U_i \subset U_j$, then there exist a smooth embedding $\rho_{ij} \colon {\tilde U}_i \to {\tilde U}_j$
 such that $\varphi_i = \varphi_j \circ \rho_{ij}$.
 \end{enumerate}
 \end{definition}
 
As smooth actions of finite groups are locally linearizable, any
orbifold has an atlas consisting of {linear charts}, that is charts where  $\Gamma_i$ acts on $\RR^n$ via an orthogonal representation
$\Gamma_i < \OO(n)$.

For any point $x\in X$, take an orbifold chart $(U,\tilde U,\Gamma,\varphi)$ with $\varphi(0)=x$. Then
we call $\Gamma$ the \emph{isotropy group} at $x$, and we denote it by $\Gamma_x$ and write
$m(x)=|\G_x|$, which is called isotropy coefficient or multiplicity of $x$. 
We call $x \in X$ a \emph{regular} point if the isotropy group $\G_x=\{ 1 \}$ is trivial
(that is $m(x)=1$), and an \emph{isotropy} point if it is not regular.
We call $x\in X$ a cyclic isotropy point if $\G_x$ is a cyclic group (that is $\G_x=\ZZ_{m(x)}$), and $X$ is
a \emph{cyclic orbifold} if all isotropy groups are cyclic. 
We call $x\in X$ a \emph{smooth} point if a neighbourhood of $x$ is 
homeomorphic to a ball in $\RR^n$, and singular otherwise. 
Clearly a regular point is smooth, but not conversely. 

An orbifold $X$, with atlas $\{(U_i, {\tilde U}_i, \Gamma_i, \varphi_i)\}$, 
 is {\em oriented} if each $ {\tilde U}_i$ is oriented, the action of $\Gamma_i$ is 
 orientation-preserving, and all the change of charts 
 $\rho_{ij} \colon {\tilde U}_i \to {\tilde U}_j$ 
 are orientation-preserving. In this case, we can arrange that $\Gamma_i <\SO(n)$.

\begin{definition}[\cite{BG}] \label{def:orbimap}
Let $X$ and $Y$ be two orbifolds with atlas
$\{(U_i, {\tilde U}_i, \Gamma_i, \varphi_i)\}$ and $\{(V_j, {\tilde V}_j, \Upsilon_j, \psi_j)\}$, respectively.
A map $f \colon X \to Y$ is said to be an {\em orbifold map} 
if $f$ is a continuous map between the underlying topological spaces, and 
for every point $p\in X$ there are orbifold charts $(U_i, {\tilde U}_i, \Gamma_i, \varphi_i)$ at $p$ and
$(V_i, {\tilde V}_i, \Upsilon_i, \psi_i)$ at $f(p)$, with $f(U_i)\subset V_i$, a 
differentiable map ${\tilde f}_i \colon {\tilde U}_i \to {\tilde V}_i$, with $\tilde f(0)=0$, and a
homomorphism $\varpi_i:\Gamma_i\to \Upsilon_i$ such that ${\tilde f}_i \circ \gamma=\varpi_i(\gamma)\circ
{\tilde f}_i$ for all $\gamma\in \Gamma_i$, and 
$f_{| U_{i}}\circ\varphi_i=\psi_i\circ{\tilde f}_i$.
 Moreover, if $\rho_{ij} \colon {\tilde U}_i \to {\tilde U}_j$ is a change of charts
for $p$, then there is a change of charts $\mu(\rho_{ij}) \colon {\tilde V}_i \to {\tilde V}_j$ for $f(p)$ such that
${\tilde f}_j\circ\rho_{ij} = \mu(\rho_{ij})\circ{\tilde f}_i$, and $\mu(\rho_{ij}\circ \rho_{ki}) =  \mu(\rho_{ij})\circ \mu(\rho_{ki})$
for changes of charts $\rho_{ki} \colon {\tilde U}_k \to {\tilde U}_i$ 
and $\rho_{ij} \colon {\tilde U}_i \to {\tilde U}_j$.
\end{definition}

The composition of orbifold maps is an orbifold map. Two orbifolds $X$ and $Y$ are said to be {\em diffeomorphic} if there exist orbifold maps
$f \colon X \to Y$ and $g \colon Y \to X$ such that
$g\circ f=1_{X}$ and $f\circ g=1_{Y}$, where $1_{X}$ and $1_{Y}$ are the respective identity maps.
Equivalently, an \emph{orbifold diffeomorphism} $f:X\to Y$ is an orbifold map such that is a homeomorphism on the underlying topological 
spaces, the maps $\tilde f_i$ are diffeomorphims, and the maps $\varpi_i$ are isomorphisms. An orbifold diffeomorphism
is orientation preserving (resp.\ reversing) if the maps $\tilde f_i$ preserve (resp.\ reverse) the orientation.

Considering  $\RR$ as an orbifold, we can define {\em orbifold functions} on an orbifold $X$ as 
orbifold maps $f \colon X \to \RR$. We denote $C^\infty_{\orb}(X)$ the set of orbifolds functions on $X$.
An \emph{orbifold partition of unity} subordinated to a locally finite cover $\{U_\a\}$ of $X$ is a collection of orbifold functions
$\{\rho_\a\}$ with $\rho_\a \geq 0$, $\sum \rho_\a \equiv 1$ and the support of $\rho_\a$ lies inside $U_\a$ for
all $\a$. By \cite[Proposition 5]{MR}, orbifold partitions of unity always exist.

\subsection{Orbivector bundles}

Now fix a finite-dimensional vector space $\RR^m$ and a finite
group $F<\OO(m)$. We call \emph{orbivector space} to the quotient $V=\RR^m/F$. 

\begin{definition}\label{def:orbibundle}
Let $X$ be a smooth orbifold, of dimension $n$, and let $\{(U_i, {\tilde U}_i, \Gamma_i, \varphi_i)\}$ be
 an atlas for $X$. An {\em orbivector bundle} over $X$ with fiber $\RR^m/F$ consists of a smooth orbifold $E$, of dimension $m+n$, and 
an orbifold map $\pi \colon  E \too  X$, 
called {\em projection}, satisfying the following conditions:
\begin{enumerate}
\item[i)] For every orbifold chart $(U_i, {\tilde U}_i, \Gamma_i, \varphi_i)$ on $X$, there exists 
an orbifold chart  $(V_i, {\tilde V}_i, \Upsilon_i, \Psi_i)$ on $E$, such that
$V_i = \pi^{-1}(U_i)$, ${\tilde V}_i={\tilde U}_i \x \RR^m$.
The action of $\Upsilon_i$ on ${\tilde U}_i  \x \RR^m$ is diagonal, that is
$\Upsilon_i < \OO(n)\x \OO(m) \subset \OO(n+m)$, $F=\Upsilon_i \cap \OO(m)$ and
$\Gamma_i=\pr_1 (\Upsilon_i)$, where $\pr_1 :\OO(n)\x \OO(m)\to \OO(n)$ is the first projection. In particular,
there is an exact sequence 
 $$
 0 \to F \to \Upsilon_i \stackrel{\pr_1}{\to} \Gamma_i \to 0.
 $$
The map
$\Psi_i \colon  {\tilde V}_i={\tilde U}_i \x \RR^m \too  V_i=E_{|U_{i}} =  \pi^{-1}(U_i)$ 
satisfies $\pi_{|V_i} \circ \Psi_i = \varphi_i \circ \pr_1$.

\item[ii)] If $(U_i, {\tilde U}_i, \Gamma_i, \varphi_i)$ and $(U_j, {\tilde U}_j, \Gamma_j, \varphi_j)$ are
two orbifold charts on $X$, with $U_i \subset U_j$, and $\rho_{ij} \colon {\tilde U}_i \to {\tilde U}_j$
is a change of charts, then there exists a differentiable map, called {\em transition map}
$g_{ij} \colon  {\tilde U}_i  \too  \GL(m)$,
and a change of charts
$\lambda_{ij} \colon {\tilde V}_i={\tilde U}_i \x \RR^m \to {\tilde V}_j={\tilde U}_j \x \RR^m$ of $E$, such that
 $$
 \lambda_{ij}(x, y) = \big(\rho_{ij}(x), g_{ij}(x)(y)\big),
 $$
for all $(x, y)\in{\tilde U}_i \x \RR^m$.
\end{enumerate}
\end{definition}

Note that if $\pi \colon  E \too  X$ is an orbivector bundle, and $x\in X$, then the {\em fiber} $\pi^{-1}(x)$
is isomorphic to $\RR^m/F_x$, where $F_x=\pr_2(\Upsilon_i)$, where
${\pr_2}:\OO(n)\x \OO(m)\to \OO(m)$ is the second projection. In particular, at a regular point 
$\Gamma_x=\{1\}$, $F_x=F$ and $\pi^{-1}(x) \cong \RR^m/F$. 

A \emph{metric} on an orbivector bundle $\pi:E\to X$ is an orbifold scalar product on every fiber $\pi^{-1}(x)=\RR^m/F_x$,
that is a scalar product on $\RR^m$ which is $F_x$-invariant, varying smoothly. For every chart as in Definition \ref{def:orbibundle}(i),
we have on $\tilde V_i=\tilde U_i\x\RR^m$ a scalar product $a(x)=\sum a_{ij}(x)y_iy_j$ on $\RR^m$, depending on $x\in \tilde U_i$, which is
$\Upsilon_i$-invariant. Using partitions of unity as in \cite[Proposition 6]{MR}, we can see that there is a metric on any
orbivector bundle. 

If $a$ is a metric on the orbivector bundle $E\to X$, we say that a chart $(V, {\tilde V}, \Upsilon, \Psi)$ is \emph{orthogonal}
if $a=\sum y_i^2$ is the standard scalar product in the chart. They always exist: just take a chart $(V, {\tilde V}, \Upsilon, \Psi)$ 
and write $a=\sum a_{ij}(x)y_iy_j$. We take the standard basis and apply the Gram-Schmidt process
to obtain an orthonormal basis $e(x)=\{e_1(x),\ldots, e_m(x)\}$. As $a$ is $\Upsilon$-invariant, the map $E:\tilde U\x \RR^m \to
\tilde U\x \RR^m$, $(x,y)\mapsto (x,e(x)(y))$ is $\Upsilon$-equivariant. The orbifold chart $(V, {\tilde V}, \Upsilon, \Psi\circ E)$ is orthogonal.
When we use an atlas consisting of orthogonal charts, the transition maps are $g_{ij} \colon  {\tilde U}_i  \too  \OO(m)$,

An oriented orbivector space $V=\RR^m/F$ consists of a vector space $\RR^m$ with an orientation,
and $F<\SO(m)$. An \emph{oriented orbivector bundle} is an orbivector bundle with oriented orbivector
space as fiber, and such that the transition maps are orientation preserving, that is $\det(g_{ij}(x))>0$ for
all $x\in \tilde U_i$. 
If $X$ is an oriented orbifold and $\pi:E\to X$ is an oriented orbivector bundle, then $E$ is an oriented
orbifold and $\Upsilon_i<\SO(n)\x \SO(m)$.

If $\pi:E\to X$, then the \emph{reverse oriented orbivector bundle}, denoted by  $\overline{E}$, is the
same orbivector bundle, endowed with the opposite orientation of the fiber orbivector space. 

Two orbivector bundles $\pi_1:E_1\to X$, $\pi_2:E_2\to X$ are isomorphic if there is an orbifold
diffeomorphism $f:E_1\to E_2$ such that $\pi_2\circ f =\pi_1$, and for each orbifold chart
$(U_i,\tilde U_i,\G_i,\varphi_i)$ of $X$, there are orbifold charts $(V_i,\tilde V_i,\Upsilon_i,\Phi_i)$
and $(W_i,\tilde W_i,\Theta_i,\Psi_i)$ of $E_1,E_2$, respectively, such that the orbifold
lift $\tilde f_i:\tilde V_i=\tilde U_i\x
\RR^m\to \tilde W_i=\tilde U_i\x \RR^m$, with $\Psi_i\circ \tilde f_i=f\circ \Phi_i$, is
of the form $\tilde f_i(x,u)=(x, \hat f_i(x)(u))$, where $\hat f_i:\tilde U_i\to \GL(m)$. Moreover, the
isomorphism $\varpi_i:\Upsilon_i \to \Theta_i$ sits in an exact sequence
 $$
    \begin{array}{ccccc} 
    0\to & F & \to \,  \Upsilon_i \, \to & \G_i & \to 0 \\ 
    & \downarrow \cong\!\!\! & \quad \downarrow \varpi_i & || \\
    0\to & F & \to \, \Theta_i \, \to & \G_i & \to 0 
    \end{array}
$$
The isomorphism $f$ is orientation preserving if $\det(\hat f_i(x))>0$ for all $x\in \tilde U_i$.
It is orientation reversing if $f:E_1\to \overline E_2$ is orientation preserving.

Finally, let $f:X' \to X$ be an orbifold diffeomorphism, and $\pi:E\to X$ be an orbivector bundle.
We define the pull-back $E'=f^*E$ by taking the total space as the usual pull-back
$E'=\{(x',v)\in X'\x E \, |\, f(x')=\pi(v)\}$, and the charts of $E'$ as follows. Take 
 orbifold charts $(U_i, {\tilde U}_i, \Gamma_i, \varphi_i)$ of $X'$ and
$(V_i, {\tilde V}_i, \Delta_i, \psi_i)$ of $X$, with $f(U_i)\subset V_i$, with lift
${\tilde f}_i \colon {\tilde U}_i \to {\tilde V}_i$, and an
isomorphism $\varpi_i:\Gamma_i\to \Delta_i$ as in Definition \ref{def:orbimap}. Let $(W_i,\tilde W_i,\Upsilon_i,\Phi_i)$
be an orbifold chart for $E$ with $\tilde W_i=\tilde V_i\x \RR^m$ and 
exact sequence $0\to F\to \Upsilon_i\to \Delta_i \to 0$, where $\Upsilon_i<\OO(n)\x \OO(m)$. Then
we define an orbifold chart for $E'=f^*E$ as 
$(Z_i,\tilde Z_i,\Theta_i,\Psi_i)$ with $\tilde Z_i=\tilde U_i\x \RR^m$ and 
$\Theta_i<\OO(n)\x \OO(m)$ defined by the pull-back exact sequence 
 $$
    \begin{array}{rcl} 
    0\to  F & \to\Theta_i  \to & \G_i  \to 0 \\ 
    || & \downarrow & \downarrow \varpi_i \\
    0\to  F & \to \Upsilon_i \to & \Delta_i  \to 0 
    \end{array}
$$

An \emph{orbifold vector bundle} over an orbifold is the case of an orbivector bundle with 
$F=\{1\}$ and fiber $V=\RR^m$ (cf.\ \cite[Definition 3.5]{BBFMT}).
The {\em orbifold tangent bundle} $TX$ of an orbifold $X$ is defined as follows.
For each orbifold chart $(U_i, {\tilde U}_i, \Gamma_i, \varphi_i)$
of $X$, we consider the tangent bundle $T{\tilde U}_i \cong{\tilde U}_i \x \RR^n$  over ${\tilde U}_i$.
Take $\rho_{i} \colon  \Gamma_{i}  \too  \GL(n)$ the homomorphism given by the action of $\Gamma_{i}$ on ${\RR}^n$. Then 
$(TX_{| U_{i}}, {\tilde U}_{i} \x \RR^n, \Gamma_i, \Psi_{i})$ is an orbifold chart for $TX$, where
$TX_{| U_{i}}  = T{\tilde U}_i/\G_i$, and $\Psi_i$ is the quotient map.
If $\rho_{ij} \colon {\tilde U}_i \to {\tilde U}_j$ is a change of charts for $X$, 
the transition map $g_{ij} \colon  {\tilde U}_i  \too  \GL(n)$ for $TX$ is given by the Jacobian matrix 
of $\rho_{ij}$. The orbifold cotangent bundle $T^* X$ and the orbifold tensor bundles are
constructed similarly. Thus, one can consider Riemannian metrics, almost complex structures, 
orbifold forms, etc. 

\begin{definition}\label{def:orbisection}
A {\em section} of an orbifold vector bundle $\pi\colon E \too  X$ is an
orbifold map $s \colon X \too  E$ such that $\pi \circ s = 1_{X}$. Therefore, if $\{(U_i, {\tilde U}_i, \Gamma_i, \varphi_i)\}$
is an atlas on $X$, then $s$ 
consists of a family of smooth maps $\{s_{i} \colon {\tilde U}_i \too  \RR^m\}$, 
such that every  $s_{i}$ is $\Gamma_i$-equivariant and compatible with the changes of charts on $X$.
\end{definition}

An (orbifold)  {\em Riemannian metric} $g$ on $X$ 
 is a positive definite symmetric tensor in $T^* X\ox T^*X$. This is equivalent to have, for each orbifold chart  
 $(U_i, {\tilde U}_i, \Gamma_i, \varphi_i)$ on $X$, a Riemannian metric 
 $g_{i}$ on the open set ${\tilde U}_i$ that is invariant under the action of $\Gamma_i$ on ${\tilde U}_i$ ($\Gamma_i$ acts on ${\tilde U}_i$
by isometries), and the change of charts $\rho_{ij} \colon {\tilde U}_i \to {\tilde U}_j$
are isometries. An (orbifold)  {\em almost complex structure} $J$ on $X$ 
 is an endomorphism $J \colon TX \to TX$ such that $J^2 = -\Id$. Thus, $J$ is determined 
 by an almost complex structure $J_i$ on ${\tilde U}_i$, for every orbifold chart  
 $(U_i, {\tilde U}_i, \Gamma_i, \varphi_i)$ on $X$, such that the action of $\Gamma_i$ on ${\tilde U}_i$
 is by biholomorphic maps, and any change of charts  $\rho_{ij} \colon {\tilde U}_i \to {\tilde U}_j$
 is a holomorphic embedding.
An {\em orbifold $p$-form} $\alpha$ on $X$ is a section of $\bigwedge^p T^* X$.
This means that, for each orbifold chart  
 $(U_i, {\tilde U}_i, \Gamma_i, \varphi_i)$ on $X$, we have
 a differential $p$-form $\alpha_i$ on the open set ${\tilde U}_i$, such that every 
 $\alpha_i$ is $\Gamma_i$-invariant (i.e. $\gamma^{*}(\alpha_i)= \alpha_i$, for  $\gamma\in\Gamma_i$),
 and any change of charts  $\rho_{ij} \colon {\tilde U}_i \to {\tilde U}_j$
 satisfies  $\rho^{*}_{ij}(\alpha_j)=\alpha_i$.
 The space of $p$-forms on $X$ is denoted by
$\Omega_{\orb}^{p}(X)$.

\subsection{Suborbifolds}
There are different notions of sub-objects in the orbifold category \cite{BB,Wei}, each of them
suitable for a different situation. We will use the following:

\begin{definition} \label{def:suborb}
Let $Z$ be an $n$-dimensional orbifold. A \emph{suborbifold} $X$ of $Z$ is a $p$-dimensional  orbifold such
that the underlying spaces $X\subset Z$, and at every point of $X$, we have a chart
$( U,\tilde U,\G,\varphi)$ of $Z$, where $\tilde U\subset \RR^n$, and a chart
$( U',\tilde U',\G',\varphi')$ of $X$, where $\tilde U' =\tilde U\cap ( \RR^p\x\{0\})$, $U'=U\cap X$, 
$\varphi'=\varphi_{|\tilde U'}$. 

A \emph{normalizable suborbifold} $X\subset Z$ is a suborbifold such that at every $x\in X$ there is an adapted
chart $( U,\tilde U,\G,\varphi)$, where $\tilde U\subset \RR^n$, such that 
$\G< \OO(p)\x \OO(n-p)\subset \OO(n)$, 
and $\G'=\pr_1(\G)<\OO(p)$, where $\pr_1:\OO(p)\x \OO(n-p) \to \OO(p)$ is the projection
on the first factor.
\end{definition}

The chart in Definition \ref{def:suborb} is called an \emph{adapted chart}. A suborbifold satisfies the
following \emph{fullness condition} \cite{Wei}: if $g\in \G$, $x\in \tilde U'$, $g x\in \tilde U'$, then there is some $h\in\G'$ with $h x=g x$.

For a normalizable suborbifold, take an adapted chart $( U,\tilde U,\G,\varphi)$. Then
all elements in $\G$ fix the subspace $\RR^p=\RR^p \x \{0\} \subset \RR^n$ 
since $\G< \OO(p)\x \OO(n-p)$. The elements in $K=\ker (\pr_1)$ 
fix pointwise $\RR^p$. Hence the induced action on $\RR^p$ is given by the group $\G'= \G/K =\pr_1(\G)$. Clearly, the fullness
condition is satisfied.
However, note that the normalizable condition is stronger, since it implies that for any
$g\in \G$, $g \tilde U' =\tilde U'$. Therefore, the preimage of $U'\subset U$ under
$\varphi:\tilde U\to U$ is exactly $\varphi^{-1}(U')=\tilde U'$.

\begin{remark}
An example of a non-normalizable suborbifold is the following. Take $\ZZ_m$ acting on
$\CC^2$ as $\xi\cdot(z_1,z_2)=(\xi z_1, \xi^l z_2)$, where $\xi=e^{2\pi i/m}$, $\gcd(l,m)=1$.
Then take $Z=\CC^2/\ZZ_m$ and $X=\CC$ the image of $z\mapsto (z,az)$, for $a\neq 0$.

Given a suborbifold $X\subset Z$, it might be that $i:X\hookrightarrow Z$ is not an orbifold map.
This happens when we do not have $\G'<\G$ in Definition \ref{def:suborb}. An example is
given by $Z=\CC^2/\ZZ_4$, where $\ZZ_4$ acts via $(z_1,z_2) \mapsto (-z_1, iz_2)$, and $X=(\CC\x\{0\})/\ZZ_2$. 
\end{remark}

\begin{proposition} \label{prop:8}
For a (connected) normalizable suborbifold $X\subset Z$, there is a well-defined \emph{normal orbivector bundle} $\nu_X$.
\end{proposition}

\begin{proof}
Consider an atlas of adapted charts $\{( U_i,\tilde U_i,\G_i,\varphi_i)\}$ that covers $X$. For $( U_i,\tilde U_i,\G_i,\varphi_i)$ we consider the
projection $\pr_1:\OO(p)\x \OO(n-p) \to \OO(p)$ and let $\G'_i=\pr_1(\G_i)<\OO(p)$. Then let $F_i=\ker (\pr_1)<\OO(n-p)$ and
the \emph{normal fiber} to be the orbivector space $V=\RR^{n-p}/F_i$. For $\tilde U'_i=\tilde U_i\cap (\RR^p\x\{0\})$,
we consider the chart $\tilde V_i = \tilde U'_i \x \RR^{n-p}$
with the action of $\G_i<\OO(p)\x\OO(n-p)\subset \OO(n)$, $V_i=\tilde V_i/\G_i$, and $\Psi_i:\tilde V_i\to V_i$ the
quotient map. Let us see that these $V_i$ glue together to give an orbivector bundle. For a change of charts $U_i\subset U_j$
with $\rho_{ij}:\tilde U_i\to \tilde U_j$, we write $\rho_{ij}=(\rho_{ij}',\rho_{ij}'')$, where $\rho_{ij}':\tilde U_i'\to \tilde U_j'$.
As $F_i$ consists of the maps that fix pointwise $\RR^p$, it is also equal to the maps that fix pointwise any open subset of $\RR^p$.
Therefore $F_i \cong F_j$ under the homomorphism $\G_i\to \G_j$. This proves in particular, using the connectedness of $X$, that
all $F_i$ are isomorphic, hence we can write $F_i=F$. Also, we have a diagram
 $$
    \begin{array}{rcl} 
    0\to  F_i & \to \, \G_i \, \to & \G_i'  \to 0 \\ 
    || & \downarrow & \downarrow \\
    0\to  F_j & \to \, \G_j\,  \to & \G_j'  \to 0 
    \end{array}
$$
The change of charts for $\nu_X$ is given by $\hat\rho_{ij}=(\rho_{ij}',d\rho_{ij}'') : \tilde V_i= \tilde U_i'\x \RR^{n-p} \to
\tilde V_j= \tilde U_j'\x \RR^{n-p}$.
The charts $\{ (\nu_{X|U_i}=V_i, \tilde V_i,\G_i,\Psi_i)\}$ give the atlas for $\nu_X$.
\end{proof}

If $Z$ is an oriented orbifold and $X$ is an oriented normalizable suborbifold then 
$\G<\SO(p)\x \SO(n-p)$. Hence $\nu_X$ is an oriented orbivector bundle.

\begin{proposition} \label{prop:tubular-nbhd}
   Let $X\subset Z$ be a connected normalizable suborbifold. 
   There exists an open neighbourhood of $X$ diffeomorphic to a neighbourhood of the zero section of
   the normal bundle $\nu_X$. If $X,Z$ are both oriented, then the diffeomorphism is orientation preserving.
\end{proposition}
 
\begin{proof}
We put an orbifold metric $g$ on $Z$. For each point $x\in X$, we take an adapted orbifold chart
$(U,\tilde U,\G,\varphi)$ with $\tilde U'=\tilde U \cap (\RR^p\x\{0\})$, $\G'=\pr_1(\G)$, under
$\pr_1:\OO(p)\x \OO(n-p)\to \OO(p)$. Take $F=\ker (\pr_1)$. We define the normal
space $\tilde\nu_x=(T_x\tilde U')^\perp\cong \RR^{n-p}$ with respect to $g_x$, for $x\in \tilde U'$, 
and consider
 $$
  \tilde \nu_{\tilde U'}=\bigsqcup_{x\in \tilde U'} \nu_x \cong \tilde V'=\tilde U' \x \RR^{n-p}
  $$
and $\nu_{X|U'}=V= \tilde\nu_{\tilde U'}/\G$, as in Proposition \ref{prop:8}. 
Then $(\nu_{X|U'},\tilde\nu_{\tilde U'},\G,\varphi)$ is a
chart for $\nu_X$ over $( U',\tilde U',\G',\varphi')$.

Now we define the exponential map $\exp^\perp:\nu_X \to Z$. For this, take
$T\tilde U=\tilde U\x \RR^n$, and consider a point 
 $$
 (x,v)=(x,0,0,v)\in \tilde\nu_{\tilde U}= T\tilde U \cap \big((\RR^p\x\{0\})\x (\{0\}\x \RR^{n-p})\big),
 $$
and use the exponential for the metric $g$ on $\tilde U$, as 
 $$
  \exp^\perp (x,v) := \exp_{(x,0)}(0,v).
 $$
This is well-defined for some $|v|< \epsilon(x)$. It is $\G$-equivariant, since the metric
is $\G$-invariant. Therefore it defines an orbifold map. It is diffeomorphism near $X$,
since $d\exp^\perp$ is an isomorphism over $v=0$. Therefore there exists a neighbourhood
$\cU=\{(x,v)| \, x\in X,v\in \nu_x, |v|<\epsilon(x)\}\subset \nu_X$ such that $\exp^\perp:\cU \to \cV\subset Z$ is a 
diffeomorphism. The function $\epsilon(x)$ can be taken continuous, and if $X$ is compact, then
we can take $\epsilon_0=\min\{\epsilon(x)| \,  x\in X\}>0$.
\end{proof}

\section{Orbifold Gompf connected sum} \label{sec:3}

\subsection{Symplectic orbifolds}
\begin{definition} \label{def:orb-sympl}
 A symplectic orbifold $(Z,\omega)$ is an oriented  orbifold $Z$ with an $\omega\in \Omega^2_{\orb}(Z)$ such that 
$d\omega=0$ and $\omega^n>0$, where $2n=\dim Z$.
\end{definition}

If $(Z,\omega)$ is a symplectic orbifold, then at every point $x\in Z$ there are \emph{orbifold Darboux charts}
\cite[Proposition 10]{MR}, that is
an orbifold chart $(U,\tilde U, \G,\varphi)$ such that $\G<\UU(n)$ and $\omega=\sum dx_i\wedge dy_i$,
in these coordinates $(x_1,y_1,\ldots, x_n,y_n)$.

Given a symplectic orbifold $(Z,\omega)$, an orbifold almost complex $J$ is  \emph{compatible}
if the orbifold tensor $g$ defined by $g(-,-)=\omega(-,J(-))$, is an orbifold Riemannian metric. They 
always exist \cite[Proposition 8]{MR}.

Let $Z$ be a symplectic $2n$-dimensional orbifold. A \emph{symplectic suborbifold} $X\subset Z$ is a 
suborbifold as in Definition \ref{def:suborb} which is a $2p$-dimensional symplectic orbifold such that there are adapted charts 
$( U,\tilde U,\G,\varphi)$ which are Darboux, where $\tilde U' =\tilde U\cap ( \RR^{2p}\x\{0\})$, and $\G<\UU(n)$, $\G'<\UU(p)$.

A \emph{symplectic normalizable suborbifold} $X\subset Z$ is a normalizable suborbifold as in Definition \ref{def:suborb}, such
that at any $x\in X$, there are adapted charts $( U,\tilde U,\G,\varphi)$, with $\omega_x$ being the standard symplectic form
of $T_xZ$, $\G< \UU(p)\x \UU(n-p)\subset \UU(n)$ and $\G'=\pr_1(\G)<\UU(p)$.
Equivalently, $\G<\SO(2p)\x \SO(2n-2p)$ and it preserves the symplectic form at $x$.

\begin{proposition}\label{prop:11}
Let $X\subset Z$ be a symplectic suborbifold. Then there is a compatible almost 
complex structure $J$ on $Z$ such that $X$ is a $J$-complex
suborbifold, that is, $J(T_xX)=T_xX$ for all $x\in X$.
\end{proposition}

\begin{proof}
We follow the argument in \cite[Proposition 8]{MR}. First, we construct
a compatible almost complex structure $J_0$ on $X$. We define the
Riemannian metric $g_0(-,-)=\omega(-,J_0(-))$ over
$X$. We extend it to the whole of $Z$ in such a way that the symplectic orthogonal
$(T_xX)^{\perp,\omega}$ to $T_xX\subset T_xZ$ is also $g_0$-orthogonal. Now 
we define the operator $A\in \End(TX)$ via $g_0(u,Av)=\omega(u,v)$, take
a square root $\sqrt{B}$ of $B=-A^2$, and define $J=-(\sqrt{B})^{-1}B$. This is 
an orbifold compatible almost complex structure. Finally note that
$A_{|T_xX}=J_0$, hence $J_{|X}=J_0$ and thus $X$ is $J$-complex suborbifold.
\end{proof}

\begin{proposition} \label{prop:12}
Let $X\subset Z$ be a symplectic normalizable suborbifold. Let $x\in X$, then there is an adapted Darboux
orbifold chart $(U,\tilde U,\Gamma, \varphi)$ at $x$ for $Z$ such that $\G< \UU(p)\x \UU(n-p)$.
\end{proposition}

\begin{proof}
Take an adapted chart $(U,\tilde U,\G,\varphi)$ such that $\tilde U'=\tilde U\cap (\RR^{2p}\x\{0\})$
and $\G<\UU(p)\x \UU(n-p)$, where $2n=\dim Z$ and $2p=\dim X$. We can assume that $\tilde U=\tilde U'\x\tilde U''$.
Take now a Darboux chart $(V',\tilde V',\G',\psi)$ for $X$ with $\tilde V'\subset \tilde U'$, and $\G'< \UU(p)$.
This gives a chart $\tilde V'\x\tilde U''$ on which $\omega_{|\tilde V'\x\{0\}}=\o_0$ is the standard form. It
can be extended to a Darboux chart on some $\tilde V'\x\tilde V''\subset \tilde V'\x\tilde U''$ 
by using the argument in the proof of \cite[Proposition 10]{MR}. Noting that $(\omega-\omega_0)_{|\tilde V'}=0$,
the form $\mu\in \Omega^1(\tilde V'\x\tilde U'')$ so that
$\omega-\omega_0=d\mu$, can be arranged to be zero over $\tilde V'$.
All the charts are the identity at first order at $x=(0,0)\in \tilde U'\x\tilde U''$, so $\G$ stays fixed. Then 
$\G<\UU(p)\x \UU(n-p)$ in the Darboux chart.
\end{proof}

A symplectic form on a orbivector space $\RR^{2m}/F$ is a non-degenerate
$2$-form on $\RR^{2m}$ invariant by $F$. In particular $F< \UU(m)$. 

\begin{definition} \label{def:13}
A \emph{symplectic orbivector bundle} $\pi:E\to X$ is 
an orbivector bundle with fiber a symplectic orbivector space
$V=\RR^{2m}/F$, as in Definition \ref{def:orbibundle} where we require $\Upsilon_i<\OO(2n)\x \UU(m)$ in (i),
and the transition maps in (ii) satisfy $g_{ij}:\tilde U_i\to \UU(m)$.
\end{definition}

In particular, a symplectic orbivector bundle $\pi:E\to X$ has a well-defined symplectic form on each fiber
$\omega_x \in \Omega^2_{\orb}(E_x)$, for $x\in X$, which identifies canonically to the orbifold
symplectic form $\omega_V$ on $V=\RR^{2m}/F$.

\begin{proposition}\label{prop:Omega-nu}
Let $(X,\omega)$ be a symplectic orbifold and let $E\to X$ be a symplectic orbivector bundle. 
 Then there exists a closed orbifold $2$-form $\Omega$ such that $\Omega_{|X}= \omega$
 and $\Omega_{|E_x}=\omega_V$, for all $x\in X$, and such that $\Omega$ is symplectic in a neighbourhood
 of the zero section. Moreover, $E_x$ and $T_xX$ are symplectically orthogonal at the 
 zero section. 
 \end{proposition}

\begin{proof}
Take a covering $\{(U_i,\tilde U_i, \G_i, \varphi_i)\}$ by orbifold Darboux charts for $X$, with
$\G_i<\UU(n)<\SO(2n)$, and so that $\tilde V_i=\tilde U_i \x \CC^m$, where $\Upsilon_i <\SO(2n)\x\UU(m)$.
As $\G_i\subset \UU(n)$, we have $\Upsilon_i<\UU(n)\x \UU(m)$.
Let $\omega_V$ be the natural symplectic form of $V=\CC^m/F$. Then take an orbifold partition of 
unity  $\{\rho_i\}$, and write $\omega_V=d\lambda$, where $\lambda=\sum x_jdy_j$, in the 
coordinates $z_j=x_j+iy_j$ of $\CC^m$.
We take $\lambda_i=(\Psi_i)_*(\sum x_j dy_j)$ and define
 \begin{equation}\label{eqn:xs}
 \Omega=\pi^*\omega+d(\sum \rho_i \lambda_i).
  \end{equation} 
Thus $\Omega_{|E_x}=\omega_V$ at every fiber. At a point $x\in X$ in the zero section, we have $\lambda_i(x)=0$, hence
$\Omega_x=\omega_X+\sum \rho_i  \omega_V=\omega_X+\omega_V$.
So $\Omega$ is symplectic over $X$, and as this is an open condition, it is so in a
neighbourhood of the zero section.
\end{proof}

\begin{proposition} \label{prop:15}
Let $X\subset Z$ be a symplectic normalizable suborbifold. Then $\nu_X$ is a symplectic orbivector bundle.
\end{proposition}

\begin{proof}
Consider a covering by Darboux adapted orbifold charts $\{(U_i,\tilde U_i, \G_i, \varphi_i)\}$ given by Proposition \ref{prop:12}. Then 
$\tilde U'_i=\tilde U_i \cap (\RR^{2p}\x\{0\})$, $\G_i<\UU(p)\x \UU(n-p)$ and
$\G'_i=\pr_1(\G_i)$. Following the construction of $\nu_X$ in Proposition \ref{prop:8}, we take $\tilde V_i=
\tilde U'_i \x \RR^{2n-2p}$, $\G_i<\SO(2p)\x \UU(n-p)$, $V_i=\tilde V_i/\G_i$, and the quotient map $\Psi_i:\tilde V_i\to V_i$.
The chart $(V_i,\tilde V_i,\G_i,\Psi_i)$ satisfies the conditions of Definition \ref{def:13}.
\end{proof}

\begin{proposition} \label{prop:16}
Let $X\subset Z$ be a (compact) symplectic normalizable submanifold. Then there is an orbifold symplectomorphism
$f:\cU\to \cV$, where $\cU\subset \nu_X$ is a neighbourhood
of the zero section, $\cV\subset Z$ is a neighbourhood of $X$, and $f_{|X}=\id_X$.
\end{proposition}

\begin{proof}
This is an extension of the argument in \cite[Proposition 19]{MR} to the case of symplectic normalizable suborbifolds.
 By Proposition \ref{prop:tubular-nbhd}, there is a diffeomorphism $f:\cU\to \cV$, where $\cU\subset \nu_X$ 
 is a neighbourhood of the zero section, and $\cV\subset Z$ is a neighbourhood of $X$ and $f_{|X}=\id$.
 Take $\Omega$ constructed in
 Proposition \ref{prop:Omega-nu}, and consider $\omega_0=\Omega$ and the pull-back $\omega_1=f^*\omega$ on $\cU\subset \nu_X$. 
  Making $\cU,\cV$ smaller if
 necessary, we can suppose that both are symplectic. Note also that $\Omega_{|X}=\omega_{|X}$
 and $\Omega_{|\nu_x}=\omega_{|\nu_x}$, for $x\in X$. As $T_xX\oplus \nu_x=T_xZ$ is a
 symplectic orthogonal splitting for both symplectic forms, 
 we have that $\omega_0=\Omega$ and $\omega_1=f^*\omega$ coincide along
 $X$. Now $[\omega_1]=[\omega_0]$ since $H^2(\cU)=H^2(X)$ (by a radial retraction), so there is an orbifold 
 $1$-form $\lambda$ such that $\omega_1-\omega_0=d\lambda$.
 We can even assume that $\lambda_x=0$ for
 $x\in X$ as in \cite[Proposition 19]{MR}.  Then Moser's trick works: take $\omega_t=\omega_0+t(\omega_1-\omega_0)$,
a vector field $X_t$ such that  $i_{X_t}\omega_t=-\lambda$, and the flow $\varphi_t:\cU \to \cU$ of $X_t$ (reducing $\cU$ if necessary). It is
standard to check that $\varphi_t^*\omega_t=\omega_0$ and $\varphi_{t|X}=\id_X$. The sought symplectomorphism is
$g=f\circ \varphi_1:\cU \to \cV$.
\end{proof}

\subsection{Reversing the normal bundle}
Now we consider the case of an \emph{orbiline bundle} $\pi:E\to X$, that is
a rank $m=1$ complex orbivector bundle, with fiber $V=\CC/F$, where $F=\ZZ_m< \UU(1)$, for some integer $m\geq 1$.
There is a natural (orbifold) hermitian metric on $E$, that defines a radial function $r$. We endow $E$ with the symplectic form
$\Omega$ given by Proposition \ref{prop:Omega-nu}. 

In a local chart $(U_i,\tilde U_i, \G_i, \varphi_i)$ for $X$, we have coordinates $x=(x_1,\ldots, x_{2n})$ for the base
and coordinates $(u,v)$ for the fiber. We will use polar coordinates $(r,\theta)$ for the fiber. 
Consider the chart $(V_i,\tilde V_i, \Upsilon_i,\Psi_i)$ for $E_{|U_i}$.
Then $\tilde V_i=U_i \x \CC$, and the product symplectic form is $\omega+du\wedge dv=\omega + rdr\wedge d\theta_i=
\omega+ d(\frac12r^2d\theta_i)$. Note that $r$ is globally well-defined, but $\theta_i=(\Psi_i)_*\theta$ is defined only
on the chart (it is defined up to addition of a function on the base). Hence we can set as in (\ref{eqn:xs}),
 $$
  \Omega=\omega+\sum \frac12 d(\rho_i  \, r^2d\theta_i).
  $$
Such a form is locally written in a chart as  
 \begin{equation}\label{eqn:b1}
 \Omega= \omega+ r^2 \alpha + \beta \wedge rdr 
  + rdr\wedge d\theta,
  \end{equation}
where $\alpha$ is a $2$-form on the base and $\beta$ is a $1$-form on the base.

Let $\overline{E}$ be the conjugate vector bundle. This is defined as $\overline{E}=E$ with the opposite
orientation. It has charts $\tilde V_i=\tilde U_i \x \CC$, where we take polar coordinates $(r',\theta_i')$
given by $r'=r$, $\theta_i'=-\theta_i$. Hence the symplectic form on the fiber is $r'dr'\wedge d\theta'=-rdr\wedge d\theta$. The process above
serves to construct a symplectic form $\overline\Omega$ on $\overline{E}$ which is
 $$
 \overline\Omega=\omega + \sum \frac12 d(\rho_i (r')^2d\theta'_i)  =\omega - \sum \frac12d(\rho_i (r')^2d\theta_i).
  $$ 
On the same trivialization as above, we have
 \begin{equation}\label{eqn:b2}
 \overline\Omega= \omega - (r')^2 \alpha - \beta \wedge r'dr' 
 - r'dr'\wedge d\theta.
 \end{equation}

We denote by $B_\e(E)=\{e\in E\,|\, r(e)<\e\}$ the $\e$-disc bundle, and $S_\e(E)=B_{2\e}(E)-\overline{B_{\e}(E)}$.
Take the map 
 \begin{equation}\label{eqn:ge}
 g_\e:S_\e(E) \to S_\e(\overline{E}),
 \end{equation}
defined by
 $$
 \theta'=-\theta,  r'= \left( 5 \e^2- r^2\right)^{1/2}.
 $$
This sends $r\in (\e,2\e)$ to $r'\in (\e,2\e)$ and $r dr=-r'  dr'$. 
Therefore
 $$
 g_\e^*(\overline\Omega) = \omega - (5\e^2-r^2) \alpha + \beta \wedge rdr 
 + rdr\wedge d\theta .
 $$
Hence 
  $$
   |\Omega- g_\e^*(\overline\Omega)|=|5\e^2\alpha| =O(\e^2)
 $$
on $S_\e(E)$. 
 
The map $g=g_\e:S_\e(E) \to S_\e(\overline{E})$ is the identity on cohomology, 
hence $[\Omega]=g^*[\overline\Omega]$. 
Let $\gamma=\Omega-g^*\overline\Omega$. The slice $Y=\{e\in E\, |\, r(e)=\e\}$ is a compact orbifold, and 
$S_\e(E)\cong Y \x (\e,2\e)$. Write $\gamma=\gamma_0\wedge dr +\gamma_1$,
and take $\delta=- \int_{\e}^r \gamma_0 dr$. As $d\gamma=0$, we have $d\gamma_0=-\frac{\bd \gamma_1}{\bd r}$.
Then $d\delta=\gamma_0 \wedge dr - \int_\e^r (d\gamma_0)dr=
\gamma_0 \wedge dr + \int_\e^r \frac{\bd \gamma_1}{\bd r} dr=\gamma-i_Y\gamma$. Then $|\delta|=O(\e^2)$.
Next $|\gamma|=O(\e^2)$, then $|i_Y\gamma|=O(\e^2)$. In a compact orbifold we have Hodge theory as in
a compact manifold \cite{BBFMT}. We have the differential $d:(\ker d)^\perp \to \im d$ that is an isomorphism
when using Sobolev norms $W^{k,2}$.
So we have $i_Y\gamma=d\tau$, for some $\tau \in (\ker d)^\perp \subset \Omega^1(Y)$ where
$||\tau||_{W^{k,2}} \leq ||i_Y\gamma||_{W^{k-1,2}}$. So $\gamma=d\mu$ with
$\mu =\delta+\tau$ and $|\mu|=O(\e^2)$.

Now we take a function $\rho$ that $\rho\equiv 1$ for $r\leq \e$ and $\rho\equiv 0$ for $r\geq 2\e$, so $|d\rho|=O(\e^{-1})$.
We define 
 $$
 \hat\Omega  = \Omega - d (\rho\mu ).
 $$
Thus $|\hat\Omega-\Omega|=O(\e)$, hence $\hat\Omega$ is symplectic for $\e>0$ small enough.
We have that $\hat\Omega=\Omega$ on $E-B_{2\e}(E)$ and $\hat\Omega=g^*_\e(\overline\Omega)$ on $B_\e(E)$.

\subsection{Orbifold Gompf connected sum}

 Let now $Z,Z'$ be two symplectic $2n$-orbifolds. Suppose that $X\subset Z$, $X'\subset Z'$ are
  compact $(2n-2)$-dimensional symplectic normalizable suborbifolds. Let $\nu_X \to X$ and $\nu_{X'}\to X'$ be
 the normal orbivector bundles to $X$, $X'$, respectively, defined in Proposition \ref{prop:8}. By Proposition \ref{prop:15},
 they are symplectic orbivector bundles, that is orbiline bundles. 
 Suppose that we have an orbifold diffeomorphism $\Phi:X \to X'$ such that  
 $$
 \hat\Phi: \overline{\nu}_X \stackrel{\cong}{\too} \Phi^*\nu_{X'}\, .
 $$
We endow $\nu_{X'}$ with the push-forward symplectic form $\hat\Phi_*(\overline\Omega_{\nu_X})$, so that $\hat\Phi$
becomes a symplectomorphism.
  
By Proposition \ref{prop:16}, there are neighbourhoods $X\subset \cV\subset Z$, $X'\subset \cV' \subset Z'$
and neighbourhoods of the zero section $\cU\subset \nu_X$, $\cU'\subset \nu_{X'}$ with orbifold 
symplectomorphisms $f:\cU\to \cV$, $f':\cU'\to \cV'$, with $f_{|X}=\Id$, $f'{}_{|X'}=\Id$.
In (\ref{eqn:ge}) we constructed a symplectomorphism 
 $$
 g_\e: S_\e(\nu_X) \to S_\e(\overline{\nu}_X), 
 $$
and an orbifold symplectic form $\hat\Omega$ on some $B_{\delta}(\nu_X)$ such that
  $\hat\Omega=\Omega_{\nu_X}$ on $\nu_X-B_{2\e}(\nu_X)$ and $\hat\Omega=g^*_\e (\overline\Omega_{\nu_X})
  =g^*_\e (\hat\Phi^*(\Omega_{\nu_{X'}}))$ 
  on $B_\e(\nu_X)$. Here $\delta>0$ is small enough so that $B_\delta(\nu_X)\subset \cU$, 
  $B_\delta(\nu_{X'})\subset \cU'$,  and also $\e<\delta/2$.
  
 Consider now
  \begin{align*}
   Z^o &= Z- f(B_\e(\nu_X)), \\
   Z'^{o} &= Z'- f'(B_\e(\nu_{X'})). 
   \end{align*}
and the symplectomorphism
 $$
 \Theta = f'\circ \hat\Phi \circ g_\e \circ f^{-1}: f(S_\e(\nu_X)) \to f'(S_\e(\nu_{X'}))   .
 $$
  
\begin{definition}\label{def:Gompf}
We define the \emph{orbifold 
Gompf connected sum} of $Z$ and $Z'$ along $X\cong X'$ as the symplectic orbifold
 $$
 \widehat Z=(Z^o \cup Z'^{o})/\Theta ,
 $$
and it will be denoted $\widehat Z=Z \#_{X=X'} Z'$.
\end{definition}

\section{Cyclic $4$-dimensional orbifolds} \label{sec:4}

\subsection{Cyclic orbifolds}

A {\em cyclic} orbifold has all isotropy groups which are cyclic groups $\Gamma\cong \ZZ_m$,
and $m=m(x)$ is the order of the isotropy at $x$. Suppose now 
that $X$ is an oriented cyclic $4$-dimensional orbifold. 

Take $x\in X$ and a chart $\varphi:\tilde U\to U$ around $x$. Let 
$\Gamma=\mathbb{Z}_m<\SO(4)$ be the isotropy group. Then $U$ is homeomorhic to an open neighbourhood of 
$0\in \mathbb{R}^4/\mathbb{Z}_m$. A matrix of finite order in $\SO(4)$ is conjugate to a diagonal matrix in $\U(2)$ of
the type $(\exp(2\pi i j_1/m),\exp(2\pi i j_2/m))=(\xi^{j_1},\xi^{j_2})$, where $\xi=e^{2\pi i/m}$. Therefore we can
suppose that $\tilde U\subset \CC^2$ and $\Gamma=\ZZ_m=\la \xi\ra\subset \U(2)$ acts on $\tilde U$ as
 \begin{equation} \label{action}
   \xi \cdot (z_1,z_2) = (\xi^{j_1} z_1,\xi^{j_2}z_2).
 \end{equation}
Here $j_1,j_2$ are defined modulo $m$.
As the action is effective, we have $\gcd(j_1,j_2,m)=1$. We call $\mathbf{j}_x=(m,j_1,j_2)$ the \emph{local invariants}
at $x$.

We say that $D\subset X$ is an
isotropy surface of multiplicity $m$ if $D$ is closed, and there is a 
dense open subset $D^\circ\subset D$ which is a smooth surface
and $m(x)=m$, for $x\in D^\circ$. The \emph{local invariants} for $D$ are those of a point
in $D^\circ$, that is $\mathbf{j}_D=(m,j)$ such that locally $D=\{(z_1,0)\}$ and the action is
given by $\xi=e^{2\pi i /m}$, $\xi\cdot (z_1,z_2)=(z_1,\xi^j z_2)$.

\begin{proposition}[{\cite[Proposition 2]{MR}}] \label{prop:models}
Let $X$ be a (cyclic, oriented, $4$-dimensional) orbifold and $x\in X$ with local model $\CC^2/\ZZ_m$. Then there
are at most two isotropy surfaces $D_i$, with multiplicity $m_i | m$, through $x$. If there are two such
surfaces $D_i,D_j$, then they intersect transversely and $\gcd(m_i,m_j)=1$. The fundamental group of the link of
$x$ has order $d$ with $m_im_j  d=m$.
\end{proposition}

For an action given by (\ref{action}), we 
set $m_1=\gcd(j_1,m)$, $m_2=\gcd(j_2,m)$. Note that $\gcd(m_1,m_2)=1$, so we can
write $m_1m_2 d=m$, for some integer $d$. 
Put $j_1=m_1 e_1$, $j_2=m_2 e_2$, where $\gcd(e_1,m_2)=\gcd(e_2,m_1)=1$. 
Let $\eta=\xi^d$, so $\la \eta\ra=\ZZ_{m_1m_2}$. The action is given by 
$\eta \cdot (z_1,z_2) = (\exp(2\pi i e_1/m_2) z_1, \exp(2\pi i e_2/m_1)z_2)$. Therefore
$\CC^2/\ZZ_{m_1m_2}\cong \CC/\ZZ_{m_2} \times \CC/\ZZ_{m_1}$, which is homeomorphic to a ball in 
$\CC^2$ via the map $(z_1,z_2)\mapsto (w_1,w_2)=(z_1^{m_2},z_2^{m_1})$. The points of 
$D_1=\{(z_1,0)\}$ and $D_2=\{(0,z_2)\}$ define two surfaces intersecting transversely, and with multiplicities
$m_1,m_2$, respectively. 

Now $\xi$ acts on $\CC^2/\ZZ_{m_1m_2}\cong \CC^2$ by the formula $\xi \cdot(w_1,w_2)=(\xi^{m_2j_1} w_1,\xi^{m_1j_2}w_2)=
(\exp(2\pi ie_1/d) w_1, \exp(2\pi ie_2/d) w_2)$, where $\gcd(e_1,d)=\gcd(e_2,d)=1$. Therefore 
$\CC^2/\la \xi\ra \cong ( \CC/\ZZ_{m_2} \times \CC/\ZZ_{m_1})/\ZZ_d$,
the point $x$ has as link a lens space ${S}^3/\ZZ_d$, and the images of $D_1$ and $D_2$ are the points
with non-trivial isotropy, with multiplicities $m_1,m_2$, respectively. 

The immersion $D_1\subset Z$ is an oriented normalizable suborbifold. This follows 
since $\G<\SO(2)\x \SO(2)=\UU(1)\x \UU(1)<\UU(2)$. The group $F=\ker (\pr_1)=\ZZ_{m_1}$, so
the multiplicity of $D_1$ is $m_1$, that is the normal fiber is $\CC/\ZZ_{m_1}$. The point $x=(0,0)$ has
multiplicity $m_2d$ in the orbifold $D_1$, since its isotropy group is $\G'=\pr_1(\ZZ_m)=\ZZ_{m_2d}$.
Note that there is an exact sequence
 $$
    \begin{array}{ccccc} 
   \ZZ_{m_1} &\to &\ZZ_m &\stackrel{\pr_1}{\too} & \ZZ_{m_2d} \\
    \wedge & &\wedge & &\wedge \\
   \{1\}\x \UU(1) & \to &\UU(1)\x \UU(1) &\to & \UU(1) \x \{1\}
   \end{array}
 $$
The extension class of the above diagram is controlled by a map $\ZZ_{m_2d}\to \UU(1)/\ZZ_{m_1}$.

\begin{remark} \label{rem:z2}
 Suppose that the fiber is $V=\CC$ and that the isotropy points $x\in D$ are of order $2$. Then 
 there are two maps $\ZZ_2 \to \UU(1)$. The trivial map corresponds to the fact that $D'=\pi^{-1}(x)$
 is an isotropy surface of order $2$ and $D$ intersects it transversally. In particular, the orbifold is smooth.
 The map $\ZZ_2 \to \UU(1)$ sending the generator to $-1$ corresponds to a point $x$ with model
 $\CC^2/\pm 1$, which is an ordinary double point singularity, and $D$ is a surface through it.
\end{remark}

\subsection{Singular symplectic $4$-manifolds}
In \cite{MRT} we constructed smooth cyclic $4$-orbifolds starting with a symplectic $4$-manifold with embedded
symplectic surfaces intersecting symplectically orthogonally. Here we shall extend the result to construct
non-smooth cyclic orbifolds starting with a cyclic singular symplectic $4$-manifold.

\begin{definition}
A \emph{cyclic singular symplectic} $4$-manifold is a symplectic cyclic $4$-orbifold $X$ 
whose isotropy set is of dimension zero (that is, a finite set $P$ of points, called the singular set).
\end{definition}

For a cyclic singular symplectic $4$-manifold, a \emph{singular point} is an isolated isotropy point
$x\in P\subset X$. A local model around $x$ is of the form $\CC^2/\ZZ_d$, where $\xi=e^{2\pi i/d}$ acts
as 
 \begin{equation}\label{eqn:Zd}
  \xi\cdot (z_1,z_2)=(\xi^{e_1} z_1, \xi^{e_2} z_2), 
  \end{equation}
where $\gcd(e_1,d)=\gcd(e_2,d)=1$. We will denote $d(x)=d$.

First fix some further notation. A \emph{sing-symplectic surface} is a symplectic $2$-orbifold $D\subset X$ such
that if $x\in D$ is a singular point, then $D$ is fixed by $\G_x$. Two sing-symplectic surfaces $D_1,D_2\subset X$
\emph{intersect nicely} if at every intersection point $x\in D_1\cap D_2$ there are adapted
Darboux coordinates $(z_1,z_2)$ at $x$ such that $D_1=\{(z_1,0)\}$ and $D_2=\{(0,z_2)\}$ in a model
$\CC^2/\ZZ_m$, where $\ZZ_m<\UU(2)$. Therefore, if the point $x\in X$ is smooth, we recover the
notion that $D_1,D_2$ intersect symplectically orthogonally and positively.

\begin{proposition} \label{prop:smooth->orb}
 Let $X$ be a cyclic singular  $4$-manifold with set of singular points $P$. Let $D_i$ be 
 embedded sing-symplectic surfaces intersecting nicely, and 
take coefficients $m_i>1$ such that
$\gcd(m_i,m_j)=1$ if $D_i$, $D_j$ intersect. Then there is an orbifold $X$ with isotropy surfaces $D_i$ of
multiplicities $m_i$, and singular points $x\in P$ of multiplicity $m=d\prod_{i\in I_x} m_i$, $d=d(x)$, $I_x=\{ i\, | \, x\in D_i\}$.
\end{proposition}

\begin{proof}
We start by fixing a Riemannian metric on a neighbourhood of the points of $P$.
For each $x\in P$ we consider a chart $\CC^2/\ZZ_d$ in such a way that if there are (either one or two)
$D_i$'s going through it, then they are the image of $z_1=0$ or $z_2=0$. We fix the standard metric on
a neighbourhood of $x$. Then we extend it to the whole of $X$.

In \cite[Proposition 4]{MR} we construct an orbifold structure on the points of $X-P$. Let us show how
to construct it around the points of $P$ in such a way that it is compatible with the structure on $X-P$.
If $x\in P$ does not lie in any $D_i$, just consider an orbifold chart $(U, B^4_\epsilon(0),\ZZ_d,\varphi)$, where
the action of $\ZZ_d$ is given as in (\ref{eqn:Zd}).
If $x$ lies in only one $D=D_i$ with $m=m_i$, we construct the orbifold chart as follows. 
We start with the chart $\varphi:B^4_\e(0)\to U$, $B^4_e(0)/\ZZ_d \cong U$, where $D\cap U=\{(z_1,0)\}$.
We write $B_\e^4(0)=B_\e^2(0)\x B_\e^2(0)$, and consider the map
 $$
  \psi:\tilde U=B_\e^2(0)\x B_\e^2(0)  \to U, \quad \psi(z_1,z_2)=\psi(z_1,r_2e^{2\pi i \theta_2})
 =\varphi (z_1, r_2e^{2\pi i m \theta_2})
 $$
and the action of $\xi=e^{2\pi i/md}$ given by $\xi\cdot (z_1,z_2)=(e^{2\pi i e_1/d} z_1, e^{2\pi i e_2/md} z_2)$. This 
gives our chart $(U,\tilde U,\ZZ_{md},\psi)$.

If $x\in P$ lies in the intersection of two surfaces, say $D_1,D_2$, with coefficients $m_1,m_2$, then $\gcd(m_1,m_2)=1$,
by assumption. Take small neighbourhoods $V_1\subset D_1$, $V_2\subset D_2$ of $x$, which we identify
with balls $B^2_\epsilon(0)\subset \RR^2$. Consider a chart 
$\varphi:B_{\epsilon}^4(0)=B_{\epsilon}^2(0)\x B_{\epsilon}^2 (0) \to U$, with $\varphi(0,0)=x$,
$D_1\cap U=\varphi(\{(z_1,0)\})$, $D_2\cap U=\varphi(\{(0,z_2)\})$, 
and let $g$ be the standard metric on $U$. We define the orbifold chart as follows:
consider $\tilde U=B_{\epsilon}^2(0) \x B_{\epsilon}^2(0)$ and
 $$
  \psi: \tilde U \to U, \quad \psi(z_1,z_2) = \psi(r_1e^{2\pi i\theta_1}, r_2e^{2\pi i\theta_2})
 = \varphi(r_1e^{2\pi im_2\theta_1}, r_2e^{2\pi im_1\theta_2}).
 $$
The action of $\ZZ_m$, $m=m_1m_2d$, 
is given by $\xi \cdot (z_1,z_2)=(e^{2\pi i e_1/m_2d} z_1,e^{2\pi i e_2/m_1d}  z_2)$, where $\xi=e^{2\pi i/m}$.
Then $(U,\tilde U, \ZZ_m,\psi)$ is our chart at $x$.

These charts are compatible with the charts constructed in \cite[Proposition 4]{MRT} for
the smooth part of the manifold $X-P$.
\end{proof}

\begin{proposition} \label{prop:orb->symp}
Let $X$ be a singular symplectic cyclic $4$-manifold with sing-symplectic surfaces 
$D_i$ intersecting nicely, and take $m_i>1$ such that
$\gcd(m_i,m_j)=1$ if $D_i, D_j$ intersect. 
Then there is a cyclic symplectic orbifold $X$ with 
isotropy surfaces $D_i$ of multiplicities $m_i$.
\end{proposition}

\begin{proof}
Let $P\subset X$ be the finite set of singular points of $X$. The argument of \cite[Proposition 7]{MRT}
shows how to construct an orbifold symplectic form on $X-P$. We want to extend the argument to the
points of $P$.
For each singular point $p\in P$, we do as follows. If it does not belong to any $D_i$, we just take a 
Darboux chart $\varphi:B^4_\epsilon(0)\to U=B^4_\epsilon(0)/\ZZ_m$. If there is one $D_i$ through $p$, we take a
Darboux chart $\varphi:B^4_\epsilon(0)\to U=B^4_\epsilon(0)/\ZZ_m$ so that $D_i=\varphi(\{(z_1,0)\}$.
If $p$ lies at an intersection $D_i\cap D_j$, fix a Darboux chart $\varphi: B^2_{\epsilon}(0) \x
B^2_{\epsilon}(0)\to U$ with $D_i\cap U=\varphi(\{(z_1,0)\}$, 
$D_j\cap U=\varphi(\{(0,z_2)\}$. 
Take a standard metric on $U$, and the corresponding almost complex structure $J_U$ on $U$. 
The rest of the procedure is as in \cite[Proposition 7]{MRT}. We extend $J_U$ to 
compatible almost complex structures $J_i$ on each $D_i$, then take Riemannian metrics
on the normal bundles $\nu_{D_i}$ compatible with its symplectic structure, and extend this metric $g$ 
to the whole of $X$ compatible with the symplectic form. 
This produces an orbifold almost K\"ahler structure on the whole of $X$ for which each $D_i$ is a $J$-invariant surface.

Now we use this metric $g$ for producing the atlas of Proposition \ref{prop:smooth->orb} that gives
$X$ the structure of a cyclic orbifold. Let us now construct the orbifold symplectic form. We only have to do it
around the points of $P$, since on $X-P$ the construction is given in \cite[Proposition 7]{MRT}.

Let $x\in P$, and let $U$ be a neighbourhood of $x$ as above. Take coordinates $(w_1,w_2)$,
$w_1=r_1e^{2\pi i\theta_1}$, 
$w_2=r_2e^{2\pi i\theta_2}$, for the singular symplectic
manifold $X$, and so the orbifold coordinates are 
$z_1=r_1e^{2\pi i\vartheta_1}$,  $z_2=r_2e^{2\pi i\vartheta_2}$, with $\theta_1=m_2\vartheta_1$, $\theta_2=m_1
\vartheta_2$ (it may be $m_1$ or $m_2$ equal to $1$, in case there are less than two sing-symplectic surfaces through $p$). 
Thus $\omega=r_1\, dr_1\wedge d\theta_1+r_2\, dr_2\wedge d\theta_2$. We set
 $$
  \hat\omega= m_2 r_1\, dr_1\wedge d\vartheta_1+ m_1 r_2\, dr_2\wedge d\vartheta_2\, ,
 $$
which defines an orbifold symplectic form on $U$. This pastes well with the orbifold 
symplectic form $\hat\omega$ constructed for $X-P$ in \cite[Proposition 7]{MRT}.
\end{proof}

\subsection{Local invariants} \label{sec:loc-inv}
Take a singular point $x\in X$ with local invariants ${\mathbf{j}}_x=(m_1m_2d,j_1,j_2)$ with $j_1=m_1e_1,
j_2=m_2e_2$, where $\gcd(e_1,m_2d)=\gcd(e_2,m_1d)=1$. Here $e_1,e_2 \pmod{d}$.
Let $D_1=\{(z_1,0)\}$ be one of the isotropy surfaces, with coefficient $m_1$. 
This is a symplecticcally normalizable suborbifold. Its symplectic normal orbivector bundle has fiber $\CC/\ZZ_{m_1}$.
The action is
given by the group $\la\eta\ra=\la \xi^{m_2d}\ra\cong \ZZ_{m_1}$, since
 $$
  \eta\cdot (z_1,z_2)= (\xi^{j_1m_2d} z_1, \xi^{j_2 m_2d} z_2)=(z_1 , e^{2\pi i j_2/m_1}z_2).
  $$
Therefore the local invariant of $D_1$ is ${\mathbf{j}}_{D_1}=(m_1,j_2)$, where $j_2 \pmod{m_1}$.
This gives the \emph{compatibility conditions} of the local invariants for singular points and isotropy surfaces.

\begin{definition}
 Let $X$ be a cyclic $4$-orbifold with singular points $P$ and isotropy surfaces $D_i$, $i\in I$.
 We say that $\{{\mathbf{j}}_{x},{\mathbf{j}}_{D_i} \,| \, x\in P, i\in I\}$ are local invariants for $X$
 if they satisfy the compatibility conditions above.
\end{definition}

\begin{proposition} \label{prop:acabando}
Suppose that $X$ is a symplectic $4$-orbifold and the isotropy surfaces $D_i$ are disjoint. Take integers
$j_i$ with $\gcd(m_i,j_i)=1$ for each $D_i$. Then there exist local invariants for $X$.
\end{proposition}

\begin{proof}
We only need to see that if $D$ is an isotropy surface with local invariants ${\mathbf{j}}_{D}=(n,j)$,
and $x\in D$ is a singular point with $d=d(x)$, then we can assign local invariants 
${\mathbf{j}}_x=(m_1m_2d,j_1,j_2)=(nd,j_1,j_2)$ in a compatible way. 

First note that in the local model (\ref{eqn:Zd}) we can use the
generator $\xi'=\xi^{e_2}$ instead of $\xi$. That means that we can there is a model
around $x$ of the form $\CC^2/\ZZ_d$ with $D_1=\{(z_1,0)\}$ and
$\xi'\cdot (z_1,z_2)=(\xi^{e}  z_1, \xi z_2)$, where $e=e_1e_2^{-1}\in \ZZ_d^*$.

We take $d=\Pi \,  p_i^{a_i}$, $j=\Pi \, q_j^{b_j}$ the decomposition on prime numbers. Let 
 \begin{align*}
 m_1 &= n,   && m_2=1, \\
 j_2 &= j+n x, && \text{where $x=\Pi \, p_i/\gcd(\Pi \, p_i,\Pi \,  q_j)$}, \\
 e_2 &= j_2 , && e_1 = e\cdot e_2,  \qquad \qquad \qquad  j_1 = n e_1\, . 
 \end{align*}
Clearly $j_1=m_1e_1$, $j_2=m_2e_2$, $\gcd(e_1,m_2)=1$, $\gcd(e_2,m_1)=\gcd(j_2,n)=\gcd(j,n)=1$.
Now let us see that $\gcd(j_2,d)=1$. Take a prime $p|d$. If $p|j$ then $p\not| x$ and $p\not| j_2$; if
$p\not| j$ then $p|x$ and $p\not|j_2$. Now $\gcd(j_1,m_1d)=\gcd(m_1e_1,m_1d)=m_1
\gcd(ee_2,d)=m_1\gcd(j_2,d)=m_1$, and $\gcd(j_2,m_1d)=1=m_2$. These are all the conditions
to match the cyclic orbifold singularity. 

The compatibility condition reads as $m_1=n$, $j_2\equiv j \pmod{m_1}$  
and $j_1=m_1$, $j_2=e_2$, $e_1=e_2e$, and is clearly satisfied.
\end{proof}

\section{A cyclic symplectic $4$-orbifold with disjoint symplectic surfaces} \label{sec:5}

\subsection{First building block}
Consider a genus $2$ complex surface $\S_2$ with an involution $\s:\S_2 \to \S_2$ with two fixed points.
Then the quotient $C=\S_2 /\la \s\ra$ is a complex torus and the projection $\pi:\S_2  \to C$ is 
a degree $2$ ramified covering, ramified over two points $p_1,p_2$. Second, consider a complex torus $\S_1$ with
an involution $\tau: \S_1\to \S_1$ with four fixed points. The quotient $C'=\S_1/\la \tau\ra$ is a $2$-sphere and
the projection $\pi':\S_1 \to C'$ is a degree $2$ ramified covering with four ramification points
$q_1,q_2,q_3,q_4$. Then take 
  $$
  Y= (\S_2 \x \S_1) /\la \s\x\tau\ra,
  $$
which is a singular complex surface with $8$ double points of order $2$, namely $(p_i,q_j)$, $i=1,2$, $j=1,2,3,4$.
Denote $\varpi: \S_2\x \S_1\to Y$ the projection.

We have induced projections $\pi: Y \to C$, $\pi': Y \to C'$. The generic fibers of $\pi$ are torus 
$T_p=\pi^{-1}(p)=\{p\}\x \S_1$. There are two exceptional fibers,
  $$
  S_i=\pi^{-1}(p_i)= (\{p_i\}\x \S_1)/\la\tau\ra ,  \quad i=1,2,
  $$
which are spheres for $i=1,2$. The generic fibers of $\pi'$ are genus $2$ surfaces $\S_q=(\pi')^{-1}(q)=\S_2\x\{q\}$,
except for
  $$
  T_j=(\pi')^{-1}(q_j)=(\S_2 \x \{q_j\})/\la \s\ra, \quad j=1,2,3,4,
  $$
which are tori. Note that the spheres $S_i$ and the tori $T_j$ intersect at the singular points 
$S_i\cap T_j=(p_i,q_j)$. Also $T_p\cap T_j$ is a transverse intersection at one point in the smooth locus.
The intersection $T_p\cap \S_q$ is however two points.

\begin{figure}[H]
\begin{center}
\includegraphics[width=10cm]{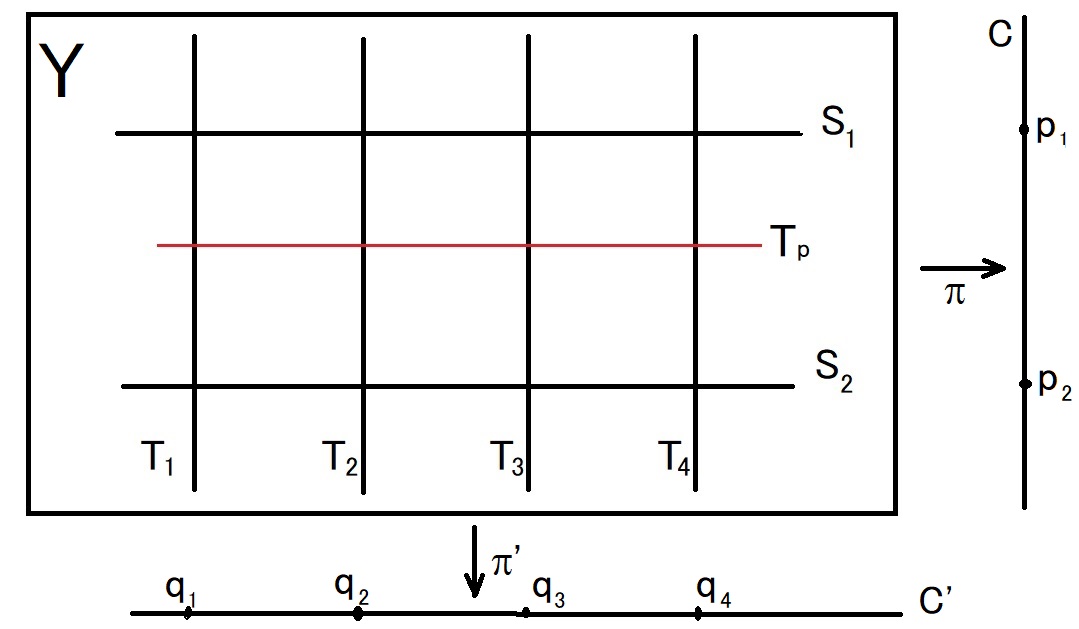}
\end{center}
\end{figure}

Let us compute the homology of $Y$.
The homology of $\S_2$ is given by $H_1(\S_2,\ZZ)=\la a_1,a_2,b_1,b_2\ra$, and the homology
of $\S_1$ given by 
 \begin{equation}\label{eqn:S11}
 H_1(\S_1,\ZZ)=\la \a,\b\ra.
 \end{equation}
 Then $\tau(\a)=-\a$ and $\tau(\b)=-\b$,
and $\s(a_1)=a_2,\s(b_1)=b_2$. So the homology of $Y$ is the coinvariant part
of the cohomology of $\S_2 \x \S_1$, that is 
 $$
  H_1(Y,\ZZ)=\la a_1,b_1\ra.
  $$
The integral homology is generated by loops $a_1,b_1$ inside a fiber $T_1$.
Note also that the loops $\a,\b$ in a fiber $T_p$ can be pushed to a special fiber $S_i$, and then
contracted there. Furthermore, the loops $\a,\b$ can be contracted in $Y-T_1$.

For the $2$-homology, we have $H_2(\S_2\x \S_1,\ZZ)= \la \S, T\ra \oplus (H_1(\S)\ox H_1(T))$. Hence
  $$
  H_2(Y,\ZZ)=\la [T_1], [S_1], a_1\x \a, b_1\x\a, a_1\x \b,b_1\x \b\ra,
  $$
where $a_1\x \a=-a_2\x\a$, etc. Note that $2[T_j]=[\S_q]$ and $2[S_i]=[T_p]$ in homology.
This implies that the $2$-homology of $Y$ is generated
by the fibers of the projections $\pi,\pi'$ plus four extra tori 
 $$
  U_1=a_1\x \a, \, U_2= b_1\x \b, \,  U_3= b_1\x\a, \, U_4= a_1\x \b,
 $$
which intersect in pairs $U_1\cdot U_2=1$, $U_3\cdot U_4=1$, with all other intersections $U_i\cdot U_j=0$.

Therefore the Betti numbers of $Y$ are $b_0=1$, $b_1=2$, $b_2=6$, and the Euler characteristic
$\chi(Y)=4$. This agrees with the fact that $\pi: \S_2\x \S_1\to Y$ is a double covering ramified at $8$ points
of multiplicity $2$, hence $\chi(\S_2\x \S_1)=2\chi(Y)-8=0$.

We make the tori $U_1,U_2,U_3,U_4$ symplectic by using the following result:

\begin{lemma}[{\cite[Lemma 27]{MRT}}]\label{lemma:lagr-sympl} 
Let $(M,\omega)$ be a $4$-dimensional  compact symplectic manifold. 
Assume that  $[F_1],\ldots ,[F_k] \in H_2(M,\ZZ)$ are linearly independent homology classes represented by 
$k$ Lagrangian surfaces $F_1,\ldots F_k$ which intersect transversely and not three of them intersect in
a point. Then there is an arbitrarily small perturbation $\omega''$ of the symplectic form $\omega$ such that 
all $F_1,\ldots,F_k$ become symplectic.
\end{lemma}

We are going to consider $S_1, S_2$ and a collection of eleven generic fibers, eight of which we denote 
$T_{p_i}=\pi^{-1}(p_i)$, for distinct points $p_i\in C-\{p_1,p_2\}$, $i=3,\ldots, 10$, 
and another three $T_{p_1'},T_{p_1''},T_{p_2'}$ for another three points $p_1',p_1'',p_2'$.
All of them are tori.


\subsection{Second building block}
We start with the complex manifold $X=\CP^2$ with a smooth complex cubic curve $C\subset \CP^2$ 
and two complex lines $L, L'$ intersecting at a point
$s_0\in C$. Each intersects $C$ at another two points.
Blow-up at $s_0$ and we obtain an exceptional divisor $E$. We denote again by $L,L',C$ the
proper transforms of $L,L',C$ in the blow-up $X'$. Now we blow-up at the point of intersection $E\cap L$. 
The resulting manifold $X''$ has a new exceptional divisor
$E'$, and we denote by $E,L,L',C$ the proper transforms of the respective curves in $X'$. Note
that now $E^2=-2$.

Next we blow-down $E$. The blow-down of a $(-2)$-rational curve is a singular double point modelled
on $\CC^2/\ZZ_2$, with the action $(z_1,z_2)\to (-z_1,-z_2)$. 
Call $X'''$ the resulting singular complex surface. 
The proper image of $L,L',C$ in $X'''$ will be called in the same way. We discard the image of $E'$
now. The curves $L',C$ intersect at the double point. Let us see that the intersection is nice.

\begin{lemma} \label{lem:symplectic orthogonal}
Let $(X,\omega)$ be a cyclic singular symplectic $4$-manifold, and suppose that $S, S' \subset X$ are sing-symplectic surfaces 
intersecting transversely and positively at a singular point $p\in X$ whose local
model is $\CC/\ZZ_d$ with local invariants (\ref{eqn:Zd}) with $e_1=e_2$. Then we can perturb $S$ near $p$ 
so that $S$ and $S'$ intersect nicely.
\end{lemma}

\begin{proof}
This follows the argument of \cite[Lemma 6]{MRT}.
We are working on $\CC^2/\ZZ_d$ with the action $\xi\cdot (z_1,z_2)=(\xi^e z_1,\xi^e z_2)$, $\xi=e^{2\pi i/d}$, $e=e_1=e_2$. 
The deformation defined in \cite[Lemma 6]{MRT} is invariant by the action of $\ZZ_d$ since the ingredient
that we introduce is a function $\rho(|z|)$ dependent only on the norm $|z|$ of $z=(z_1,z_2)$.
\end{proof}

Now blow-up $X'''$ at one of the two intersection points  of $L\cap C$ to get a
new manifold $X''''$. Let $E''$ be its exceptional divisor. Note that $L^2=0$ in $X'$, so
$L^2=-1$ in $X''$ and in $X'''$, hence now $L^2=-2$ in $X''''$. Then we can
blow-down $L$ to get a second double point modelled in $\CC^2/\ZZ_2$. Call $W'$ the resulting
manifold.

\begin{figure}[H]
\begin{center}
\includegraphics[width=12cm]{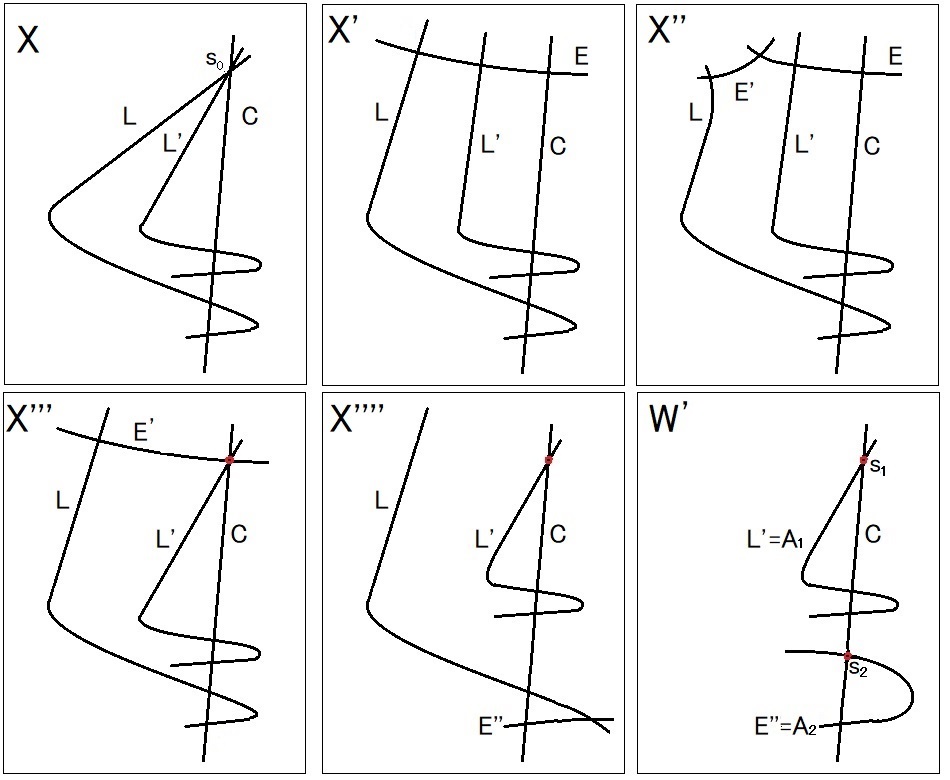}
\end{center}
\end{figure}

The final configuration is a singular complex surface $W'$
with two double points $s_1,s_2$, a genus $1$ curve $C$
through these two points, and two rational curves $A_1=L'$ and $A_2=E''$.  The first curve
$A_1$ goes through $s_1$ and intersects $C$ transversely at another two (smooth) points $s_1',s_1''$.
The second curve $A_2$ goes through $s_2$ and intersects $C$ transversely at another smooth point $s_2'$.
Using Lemma \ref{lem:symplectic orthogonal}, we make the intersection nice.

Note also that $b_2(W')=2$, since each blow-up increases $b_2$ by one, and each blow-down
decreases it by one. Also $b_1(W')=0$, hence $\chi(W')=4$. The fundamental group
does not change with the blow-ups or blow-downs, hence $W'$ is simply connected.

\begin{proposition}\label{prop:24}
 Take the loops $a,b$ generating $\pi_1(C)$. Then they can be contracted in $W'-(A_1\cup A_2)$.
\end{proposition}

\begin{proof}
 It is enough to start with a cubic curve $C$ close to a cuspidal rational curve, e.g.\ $C$ can be given by the equation
 $y^2=x^3-\epsilon^2 x$. Then the vanishing cycles lie inside a ball $B(0,2\epsilon)$. Take the lines
 $L,L'$ far away from this ball. All the process does not touch it, so the loops will still contract in $W'-(A_1\cup A_2)$.
\end{proof}

Let us compute the self-intersection of $C$ in $W'$. Note that an orbifold has a cohomology algebra
$H^*(W',\QQ)$ which has Poincar\'e duality and a well-defined intersection product over the rationals.
Next, we have the following observation: if $\pi:X_1\to X_2$ is a
blow-down map along a $(-2)$-rational curve $E$, and $A,B\subset X_1$ are divisors intersecting $E$
transversally at one point $A\cap E=\{p\}$, $B\cap E=\{q\}$, $p\neq q$, then $\bar A=\pi(A)$
and $\bar B=\pi(B)$ intersect nicely at the singular point and they satisfy 
 $$
 \text{$  \bar A \cdot \bar B= A\cdot B+\frac12$}.
  $$
Using this, we see that $C^2=8$ in $W'$. Denote $A_1\cap C=\{s_1,s_1',s_1''\}$ and $A_2\cap C=\{s_2,s_2'\}$.
We blow-up $W'$ at eight points of $C$, different
from $s_1,s_1',s_1'',s_2,s_2'$, say $s_3,\ldots, s_{10}\in C$. Call $W=W'\# 8 \overline{\CP}{}^2$
the resulting manifold, $E_3,\ldots, E_{10}$ the exceptional divisors, and note that
now $C^2=0$. Also $\chi(\tilde W')=12$.

W'e end up by noting that 
 \begin{equation}\label{eqn:A1A2}
  A_1^2 =\frac12 , \quad  A_2^2 =-\frac12.
  \end{equation}
For instance, $(L')^2=1$ in $X$, so $(L')^2=0$ in $X''$ and thus $A_1^2=\frac12$ in $W'$.


 \subsection{The connected sum}

Inside $Y$  we consider the suborbifold $T_1$, with two points of
multiplicity $2$. The fibers $T_{p_i}$ intersect  $T_1$ symplectically transversally and positively (i.e.\ nicely). The fibers
$S_1,S_2$ also intersect $T_1$ nicely at the singular points. By Remark \ref{rem:z2}, $T_1$ is a symplectically normalizable surface.
In $\tilde W$ we take the curve $C\subset W$ which is a symplectically normalizable surface.

Consider a symplectic orbifold diffeomorphism $f:T_1 \to C$, from $T_1\subset Y$ to $C\subset W$. It may be necessary
to rescale one of the sympletic forms of either $Y$ or $W$ so that the total areas of $T_1$ and $C$ are equal. Moreover,
we arrange $f$ so that it matches the two singular points lying in $T_1$ with the two singular points
in $C$, $f(p_1)=s_1, f(p_2)=s_2$. We also arrange  $f(p_i)=s_i$, $i=3,\ldots,10$ and
$f(p_1')=s_1', f(p_1'')=s_1'', f(p_2')=s_2'$. 

By Remark \ref{rem:z2} there is only one possible model for the normal bundle at the order two singular points. Moreover,
the self-intersection $T_1^2=0$ in $Y$ and $C^2=0$ in $W$ allows to have an orbivector bundle isomorphism
between the normal bundles $\nu_{T_1}$ and $\overline\nu_C$. Therefore we can 
take the orbifold Gompf connected sum
 $$
  Z'= Y \#_{T_1=C} \, W
 $$
This has  
  $$
   \chi(Z')=\chi(Y)+\chi(W)-\chi(T_1)= 16.
  $$
   
 \begin{theorem} \label{thm:25}
  The total space of the $4$-orbifold $Z'$ is simply connected. 
 \end{theorem}
  
  \begin{proof}
    The fundamental group $\pi_1(Z')$ is the amalgamated sum of $\pi_1(Y-T_1)$ and $\pi_1(W-C)$.
   The fundamental group of $Y-T_1$ is generated by two loops $a,b\in \pi_1(T_1)$. 
   They are transferred to $W-C$, where they can be contracted by Proposition \ref{prop:24}.
   The fundamental group of $W-C$ is generated by a loop around $C$. This can be contracted by using
   one of the exceptional spheres $E_i$.   
 \end{proof}
 
The Betti numbers of $Z$ are $b_0=1$, $b_1=0$ and $b_2=14$.
We have $14$ surfaces
\begin{equation} \label{eqn:V12-def}
 \begin{aligned} 
   &V_1 = A_1 \# (S_1\cup T_{p_1'}\cup T_{p_1''}), \\
   &V_2 = A_2 \# (S_2\cup T_{p_2'}), \\
   & V_i = E_i \# T_{p_i}, \quad i=3,\ldots, 10, \\
   &U_1,U_2,U_3,U_4\, .
 \end{aligned}
 \end{equation} 

They are all symplectic. 
 The genus of $V_1$ is $2$ and the genus of $V_2$ is $1$. 
 Also the surfaces $V_i$, $U_j$, $3\leq i\leq 10$, $1\leq j\leq 4$, all have genus $1$.
 
 \begin{lemma} \label{lem:4}
 If $T$, $T'$ are two symplectic tori in a symplectic $4$-manifold $X$ intersecting transversely and positively at one point, 
then the (symplectic) blow-up of $X$ contains three disjoint symplectic surfaces, two tori and one surface of genus $2$,
contained in a neighbourhood of (the blow-up of) $T\cup T'$, and linearly independent in homology.
\end{lemma}

\begin{proof}
 Let $p$ be the intersection point $T\cap T'$. Using \cite[Lemma 6]{MRT}, we can 
 take a chart $(z,w)$ around $p$ 
 where $T=\{z=0\}$, and $T'=\{w=0\}$. Consider $T+T'$ and resolve 
the singularity producing a symplectic genus $2$ surface $\Sigma$. We move it to intersect $T$ and $T'$ 
at the same point $p$. Locally the equation of $\Sigma$ is
$(z-\varepsilon)\cdot (w-\varepsilon)=\varepsilon^2$. As 
  $\Sigma\cdot T=1$, $\Sigma\cdot T'=1$, the point 
  $p$ is the only intersection point of the three surfaces $T,T',\Sigma$, and they intersect 
transversely. Moreover, $\Sigma^2=(T+T')^2=2$. Blowing up at $p$ we get a symplectic 
manifold $\tilde X=X\# \overline{\CP}{}^2$, where the proper transforms  $\tilde T,\tilde T',\tilde\Sigma$ 
are disjoint symplectic surfaces of genus
$1,1,2$ and self-intersection numbers $-1,-1,1$. 
They generate the same $3$-dimensional space in homology
as $T,T'$ and the exceptional sphere $E$. 
\end{proof}

Using Lemma \ref{lem:4}, take the orbifold 
 $$
 Z=Z'\# 2\overline{\CP}{}^2 ,
 $$
which is the blow-up twice of $Z'$ and the two points $U_1\cap U_2$, $U_3\cap U_4$. Call again $U_i$ the proper
transforms of $U_i$, and call $W_1,W_2$ the new genus $2$ symplectic surfaces. We denote $V_{10+i}=U_i$, $i=1,2,3,4$,
and $V_{15}=W_1,V_{16}=W_2$. Then $V_1,\ldots, V_{16}$ are disjoint surfaces. 

Consider coefficients $m_i$, $1\leq i\leq 16$. Applying Proposition \ref{prop:orb->symp} we put 
a symplectic orbifold structure on $Z$, so that $V_i$ is an isotropy surface with multiplicity $m_i$.

The following proves Theorem \ref{thm:main2}:

\begin{theorem} \label{thm:27}
There is a  simply connected symplectic orbifold $Z$ with $b_2=16$ and 
containing $16$ disjoint symplectic orbifold surfaces. They span $H_2(Z,\QQ)$,
thirteen surfaces are of genus $1$ and three surfaces are of genus $2$. 
\end{theorem}

The self-intersections of the surfaces are as follows:
\begin{equation} \label{eqn:V12}
 \begin{aligned} 
  V_1^2 &= \frac12, V_2^2 = -\frac12\\
  V_i^2 &=-1, i=3,\ldots, 14, \\
  V_{15}^2& =1, V_{16}^2=1.
  \end{aligned}
  \end{equation}
We only need to compute $V_1^2, V_2^2$, the other are obvious. Let us
do the case of $V_1=A_1 \# (S_1\# T_{p_1'}\# T_{p_1''})$. We have that
 $$
  V_1^2=A_1^2 + (S_1^2 +T_{p_1'}^2+T_{p_1''}^2) = \frac12,
  $$
using (\ref{eqn:A1A2}), and analogously for $V_2^2$. The formula above can be justified 
as follows. Let $(z_1,z_2)$ be coordinates on $W$ around the point $s_1\in W$ with $C=\{z_1=0\}$,
such that $A_1=\{z_2=0\}$, and $(z_1,z_2)$ coordinates on $Y$ around $p_1\in T_1$ with $T_1=\{z_1=0\}$,
such that $S_1=\{z_2=0\}$. We can perturb $A_1$ topologically as $A_1'=\{z_2=\epsilon z_1\}$ and $S_1$ 
as $S_1'=\{z_2= \epsilon \bar z_1\}$, so that they can be glued in the Gompf connected sum. The local 
intersection is $A_1\cdot A_1' =1/2$ and $S_1 \cdot S_1' =-1/2$, whereas after the connected sum $A_1\# S_1$ does
not intersect $A_1'\# S_1'$.

\section{A $5$-dimensional K-contact manifold} \label{sec:k-contact}

Consider a contact co-oriented manifold $(M,\eta)$ with a contact form $\eta$. 
We say that $(M,\eta)$ admits a {\it Sasakian structure} $(M,g,\xi,\eta,J)$ if
\begin{enumerate}
\item there exists an endomorphism $J:TM\rightarrow TM$ such that 
 $J^2=-\operatorname{Id}+\xi\otimes\eta$,
for the Reeb vector field $\xi$ of $\eta$,
\item $J$ satisfies the conditions
$d\eta(JX,JY)=d\eta(X,Y)$, for all vector fields $X,Y$ and $d\eta(JX,X)>0$ for all non-zero $X\in\ker\eta$,
\item the Reeb vector field $\xi$ is Killing with respect to the Riemannian metric $g(X,Y)=d\eta(JX,Y)+\eta(X)\eta(Y)$,
\item the almost complex structure $I$ on the contact cone 
$C(M)=(M\times\mathbb{R}_{+},t^2g+dt^2)$
 defined by 
$I(X)=J(X),X\in\ker \eta$, $I(\xi)=t{\partial\over\partial t}$, $I\left(t{\partial\over\partial t}\right)=-\xi$,
 is integrable.
\end{enumerate}
If one drops the condition of the integrability of  $I$, one obtains a {\it K-contact} structure.

A compact simply connected $5$-manifold $M$ is called a {\it Smale-Barden manifold}. 
These manifolds are classified by their second homology group over $\mathbb{Z}$ and a 
{\it Barden invariant} (see \cite[Theorem 10.2.2]{BG}). Let us write the second homology
as a direct sum of cyclic groups of prime power order
  $$
   H_2(M,\mathbb{Z})=\mathbb{Z}^k\oplus (\mathop{\oplus}_{p,i}\mathbb{Z}_{p^i}^{c(p_i)}),
   $$
where $k=b_2(M)$. Choose this decomposition in a way that the second Stiefel-Whitney class map
 $w_2: H_2(M,\mathbb{Z})\rightarrow\mathbb{Z}_2$
iz zero on all but one summand $\mathbb{Z}_{2^j}$. The value of $j$ is unique and 
it is denoted by $i(M)$ and is called the Barden invariant.

The problem of the existence of simply connected K-contact non-Sasakian compact manifolds 
is still open in dimension $5$ (open problem 10.2.1 in \cite{BG}):

{\it Do there exist Smale-Barden manifolds which carry K-contact but do not carry Sasakian structures?}

A Sasakian structure on a compact manifold $M$ is called {\it quasi-regular} if there 
is a positive integer $\delta$ satisfying the condition that each point of $M$ has a 
neighbourhood such that each leaf for $\xi$ passes through $U$ at most $\delta$ times. 
If $\delta=1$, the structure is called regular. It is known \cite{R} that if a compact manifold 
admits a Sasakian structure, it also admits a quasi-regular one. Thus, 
when we are interested in existence questions, we may consider Sasakian structures which are quasi-regular. 
Analogous results are true for K-contact manifolds \cite{MT}.

A Seifert bundle is a space fibered by circles over an orbifold. We give a precise definition.

\begin{definition} \label{definition seifert bundle}
Let $X$ be a cyclic, oriented $n$-dimensional orbifold. A Seifert bundle over $X$ is an oriented  
$(n+1)$-dimensional manifold $M$ equipped with a smooth $S^1$-action and a 
continuous map $\pi:M \to X$ such that for an orbifold chart $(U,\tilde U,  \ZZ_m,\varphi)$, there is
is a commutative diagram
 $$
\begin{tikzcd}
(\tilde U \x S^1)/\ZZ_{m} \arrow{r}{\cong} \arrow{d}{\pi} & \pi^{-1}(U) \arrow{d}{\pi} \\
\tilde U/\ZZ_{m} \arrow{r}{\cong} & U
\end{tikzcd}
 $$
where the action of $\ZZ_{m}$ on $S^1$ is by multiplication by $\xi=e^{2\pi i /m}$ and the top diffeomorphism is
$S^1$-equivariant.
\end{definition}

Let $X$ be a $4$-dimensional cyclic orbifold. Let $p$ be a cyclic isotropy point of some order $m>0$.
The local model is given by $\CC^2/\ZZ_m$ where the action is 
given for $\xi=e^{2\pi i /m}$ by
 $$
 \xi\cdot (z_1,z_2)=(\xi^{j_2} z_1, \xi^{j_1} z_2),
 $$
where $(m,j_1,j_2)\in \ZZ_m\x \ZZ_m$ are the local invariants.
The Seifert bundle is defined as
 $$
 ( \CC^2\x S^1)/\ZZ_m\, ,
 $$
with the action $ \xi\cdot (z_1,z_2,u)=(\xi^{j_2} z_1, \xi^{j_1} z_2,\xi u)$.

Let $P\subset X$ be the subset of singular points, and
$D_i$ the isotropy surfaces.

\begin{definition}
For a Seifert bundle $\pi:M\to X$, we define its Chern class as follows. Let 
$\ell=\lcm ( m(x) \, | \, x\in X)$. We denote by $M/\ell$ the quotient of $M$ by $\ZZ_\ell\subset S^1$. Then
$M/\ell \to X$ is a circle fiber bundle and it has a Chern class $c_1(M/\ell)\in H^2(X,\ZZ)$. We define
 $$
 c_1(M)=\frac{1}{\ell} c_1(M/\ell) \in H^2(X,\QQ).
 $$
\end{definition}

Let $P=\{x_j\}\subset X$ be the set of singular points. We consider small balls around all of the points, $B=\sqcup B_j\subset X$.
Every $L_j=\bd B_j$ is a lens space of order $d_j=d(x_j)$, that is $S^3/\ZZ_{d_j}$. Let $L=\sqcup L_j$.
Taking the Mayer-Vietoris exact sequence of $X-B$ and $\bar B$, and using that $H^1(L_j,\ZZ)=0$, $H^2(L_j,\ZZ)=\ZZ_{d_j}$, we have
and exact sequence 
 $$
  0 \to H^2(X,\ZZ) \to H^2(X-P,\ZZ) \to \oplus \, \ZZ_{d_j}
 $$
 Note that when $H_1(X,\ZZ)=0$, we have $H^2(X,\ZZ)=\ZZ^{b_2}$ is torsion-free, so any integral class in $H^2(X,\ZZ)$ is
 well-determined by any multiple of it.

\begin{lemma}
Let $m=\lcm ( m_i)$. 
Then $c_1(M/m)=m \, c_1(M)$ is integral in $H^2(X-P,\ZZ)$. Note that $c_1(M)=\frac1m c_1(M/m)$.
\end{lemma}

This is an easy consequence of the fact that $M/m\to X-P$ is a circle fiber bundle.
Note that $m|\ell$. So $c_1(M/m)=\frac{m}{\ell} c_1(M/\ell)$.

Let $X^o=X-B$, which is a manifold with boundary $L$. There is Poincar\'e duality 
 $$
 H^2(X^o,\ZZ) \x H^2(X^o,L, \ZZ) \to H^4(X^o,L,\ZZ).
 $$
 Now we have isomorphisms
 $H^2(X-P,\ZZ)=H^2(X^o,\ZZ)$ and $H^k(X^o,L,\ZZ)=H^k(X,B,\ZZ)=H^k(X,P,\ZZ)=H^k(X,\ZZ)$, for $k\geq 2$, since $P$ is $0$-dimensional.
 Moreover $H^4(X,\ZZ)=\ZZ$, generated by the fundamental class. Hence Poincar\'e duality is a perfect pairing
 $$
  H^2(X-P,\ZZ) \x H^2(X,\ZZ) \to \ZZ.
  $$
In particular the class $[D_i]\in H_2(X,\ZZ)$ gives a map $H^2(X-P,\ZZ)\to \ZZ$, and hence a class in $H^2(X,\QQ)$ via
the inclusion $H^2(X,\ZZ)\subset H^2(X-P,\ZZ)$.

\begin{proposition} \label{prop:14MRT}
Let $X$ be an oriented $4$-manifold and $D_i \subset X$ oriented surfaces of $X$ which
intersect transversely. Let $m_i>1$ such that $\gcd(m_i,m_j)=1$ if $D_i$ and $D_j$ intersect. 
Suppose that we have local invariants ${\mathbf{j}}_{D_i}=(m_i,j_i)$ for each $D_i$ and 
 $\mathbf{j}_x$, for every 
singular point $x\in P$, which are compatible. 

Let $0<b_i<m_i$ such that $j_ib_i\equiv 1 \pmod{m_i}$.
Finally, let $B$ be a complex line bundle on $X$. Then
there is a Seifert bundle $f:M  \to X$ with the given local invariants and first Chern class 
 \begin{equation}\label{eqn:c1}
 c_1(M)=c_1(B) + \sum_i \frac{b_i}{m_i} [D_i]. 
 \end{equation}
The set of all such Seifert bundles forms a principal homogeneous space under $H^2(X,\ZZ)$, 
where the action corresponds to changing $B$. 
\end{proposition}

\begin{proof}
 The local model of the Seifert bundle around a singular point is constructed as above, and it is determined by
 the local invariant $\mathbf{j}_x$, $x\in P$. The rest of the argument is carried out exactly as in  
 \cite[Proposition 14]{MRT}.
\end{proof}

The following is the extension of \cite[Theorem 16]{MRT} to the case of quasi-regular Seifert bundles.

\begin{theorem} \label{thm:16MRT}
Suppose that $\pi:M\to X$ is a quasi-regular Seifert bundle with isotropy surfaces $D_i$ with multiplicities $m_i$. 
Let $m=\lcm (m_i)$. Then $H_1(M,\ZZ)=0$ if and only if
 \begin{enumerate}
 \item $H_1(X,\ZZ)=0$,
 \item $H^2(X,\ZZ)\to \mathop{\oplus}\limits_i H^2(D_i,\ZZ_{m_i})$ is surjective,
 \item $c_1(M/m)\in H^2(X-P,\ZZ)$ is a primitive class.
 \end{enumerate}
 Moreover, $H_2(M,\ZZ)=\ZZ^k\oplus (\mathop{\oplus}\limits_i \ZZ_{m_i}^{2g_i})$, $g_i=$\,genus of $D_i$, $k+1=b_2(X)$.
\end{theorem}

\begin{proof}
We compute the cohomology of $M$ by using the Leray spectral sequence. Let us denote
$\cF=R^1\pi_*\ZZ_M$. By \cite[(25)]{K}, if $H_1(M,\ZZ)=0$ then $H_1(X,\ZZ)=0$. Therefore
$H^1(X,\ZZ)=H^3(X,\ZZ)=0$ and $H^2(X,\ZZ)=\ZZ^{k+1}$. The Leray spectral sequence is then
\begin{center}
\begin{tikzpicture}
  \matrix (m) [matrix of math nodes,
    nodes in empty cells,nodes={minimum width=5ex,
    minimum height=5ex,outer sep=-5pt},
    column sep=1ex,row sep=1ex]{
          \ZZ &H^1(X,\cF) &H^2(X,\cF) &H^3(X,\cF)  & \ZZ \\
          \ZZ  & 0 & \ZZ^{k+1} &  0  & \ZZ \\ 
   };
  \draw[-stealth] (m-1-2.south east) -- (m-2-4.north west);
  \draw[-stealth] (m-1-3.south east) -- (m-2-5.north west);
  \draw[-stealth] (m-1-1.south east) -- (m-2-3.north west);
\end{tikzpicture}
\end{center}
Therefore $H_1(M,\ZZ)=0$ (equivalently, $H^4(M,\ZZ)=0$) if and only if
 \begin{align*}
 & H^3(X,\cF)=0 \\  
  & H^2(X,\cF) \longrightarrow \ZZ \,\, \text{ is surjective}
 \end{align*}
By \cite[(24.2) and (24.3)]{K}, 
 \begin{align*}
 & H^2(X,\cF)= \ZZ^{k+1}\oplus (\mathop{\oplus}\limits_i\ZZ_{m_i}^{2g_i}), \\
 & H^3(X,\cF)= \coker(H^2(X,\ZZ) \to \mathop{\oplus}\limits_i H^2(D_i,\ZZ_{m_i})).
 \end{align*}
The first equality gives that $H_2(M,\ZZ)=H^3(M,\ZZ)=\ker(H^2(X,\cF)\to \ZZ)$ is as stated.

We need to understand the edge map $e:H^2(X,\cF)\to \ZZ$, given by cupping
with the Chern class $c_1(M)$ (note that this is a rational class). The exact sequence 
in \cite[(24.2)]{K} reads 
 $$
  0 \to \mathop{\oplus}\limits_i H^1(D_i,\ZZ_{m_i}) \to H^2(X,\cF) \to H^2(X,\ZZ) \to  \mathop{\oplus}\limits_i H^2(D_i,\ZZ_{m_i}) \to 0
 $$
using that $H^3(X,\cF)=0$. The map $e$ factors through the torsion-free part $\bar H^2(X,\cF)$ of 
$H^2(X,\cF)$. We have a commutative diagram
 $$
 \begin{array}{rcccccl}
  0  \to & \bar H^2(X,\cF) & \to & H^2(X,\ZZ) & \to &\oplus \, \ZZ_{m_i} & \to 0 \\
 & \, \,\, \,  \downarrow e & & \, \,\, \,  \downarrow \hat{e} & & \downarrow \\
  0 \to & \ZZ &\stackrel{m}{\to} & \ZZ & \to & \ZZ_m & \to 0
  \end{array}
  $$
where we define $\hat{e}$ by declaring $\hat{e}(s)=e(ms)$, and noting that
for any $s\in H^2(X,\ZZ)$ we have $m s\in \bar H^2(X,\cF)$.
From this exact sequence we see that $e$ is surjective if and only if $\hat{e}$ is surjective.

The map $\hat{e}$ is the edge map of the Seifert bundle $M/m$, that is cupping with $c_1(M/m)$.
Recall that
 $$
  c_1(M/m)\in H^2(X-P,\ZZ),
  $$
which defines a map $H^2(X,\ZZ)\to \ZZ$, and this coincides with $\hat e$. 
The map is surjective if and only if $c_1(M/m)$ is a primitive class.
\end{proof}

\medskip

Next we recall some results on K-contact and Sasakian manifolds that we shall use later.

\begin{proposition}[{\cite[Theorem 19]{MRT}}] \label{thm:19MRT}
Let $(M,g,\xi,\eta,J)$ 
be a quasi-regular K-contact manifold. Then the space 
of leaves $X$ has a natural structure of an almost K\"ahler cyclic orbifold where the projection $M \to X$ is a Seifert bundle.
Furthermore, if $(M,g,\xi,\eta,J)$ 
is Sasakian, then $X$ is a K\"ahler orbifold. \hfill $\Box$
\end{proposition} 

\begin{proposition}[{\cite[Theorem 21]{MRT}}] \label{prop:21MRT}
Let $(X,\omega,J,g)$ be an almost K\"ahler cyclic orbifold with 
$[\omega] \in H^2(X,\QQ)$, and let $\pi: M \to X$ be a Seifert 
bundle with $c_1(M)=[\omega]$. Then $M$ admits a K-contact structure $(M,g,\xi,\eta,J)$
such that $\pi^*(\omega)=d \eta$. \hfill $\Box$
\end{proposition}

\begin{lemma} \label{lem:20MRT}
Let $(X,\omega)$ be a cyclic symplectic
$4$-orbifold with a collection of embedded symplectic surfaces $D_i$, $i\in I$,
intersecting nicely, and integer numbers $m_i>1$, with $\gcd(m_i,m_j)=1$ 
whenever $D_i \cap D_j \neq \emptyset$. Assume that there are local
invariants $\{{\mathbf{j}}_{D_i}=(m_i,j_i), {\mathbf{j}_x}| i\in I, x\in P\}$ for $X$.
Let $b_i$ with $j_ib_i\equiv 1 \pmod{m_i}$, $m=\lcm(m_i)$. 
Then there is a Seifert bundle $\pi:M\to X$ such that:
\begin{enumerate}
\item It has Chern class $c_1(M)=[\hat\omega]$ for some orbifold symplectic form $\hat\omega$ on $X$.
\item If $\sum \frac{b_im}{m_i} [D_i] \in H^2(X-P,\ZZ)$ is primitive and the second Betti number $b_2(X)\geq 3$, then 
then we can further have that $c_1(M/m)\in H^2(X-P,\ZZ)$ is primitive.
\item If $\a$ is a given class in $H^2(X-P,\ZZ_2)$ and the image
$H^2(X,\ZZ)\to H^2(X-P,\ZZ_2)$ is at least two-dimensional, then we can also 
take $c_1(B)\not\equiv \a \pmod2$.
\item If $\a$ is a given class in $H^2(X,\ZZ)$ and $\a \not\equiv 0, 
 \frac{b_im}{m_i} [D_i] \in H^2(X-P,\ZZ_2)$, then we can also take $c_1(B) \equiv \a \pmod2$.
\end{enumerate}
\end{lemma}

\begin{proof}
This follows the arguments of the proof of \cite[Lemma 20]{MRT}. Using Proposition \ref{prop:orb->symp},
we construct a symplectic orbifold with isotropy coefficients $m_i$ for each $D_i$. Given now the 
local invariants
$\{(m_i,j_i)\}$ for the isotropy surfaces and the local invariants $\mathbf{j}_x$ for each singular point $x \in P$,
we find a Seifert bundle $M\to X$ with $c_1(M)=[\hat\omega]$ for some orbifold symplectic form
$\hat\omega$. 

The proof of (2) is as the proof of \cite[Lemma 20(2)]{MRT}, noting that the classes $b_1,b_2\in H^2(X-P,\ZZ)$
whereas the class $a_0\in H^2(X,\ZZ)$. The rest of the argument follows verbatim.

For (3), if $\a$ is not in the image of $H^2(X,\ZZ)\to H^2(X-P,\ZZ_2)$, then we are done. Otherwise, 
we assume $\a\in H^2(X,\ZZ)$. 
Take $\a'$ not proportional to $\a$ in $H^2(X-P,\ZZ_2)$. 
Now we follow the proof of \cite[Lemma 20(2)]{MRT}. We can 
take $b_1$ to satisfy also $b_1\cdot \a'=0$. Therefore we can arrange 
$a_0\not\equiv \a \pmod2$, by adding $\a'$ if necessary. 
Finally we take $k_0$ to be also even, and hence $a'$ satisfies that $a'\not\equiv\a\pmod2$.

For (4), we take $a_0$ in the proof of \cite[Lemma 20(2)]{MRT} to be the primitive
class determined by $\a$, i.e. $\a=N a_0$, $a_0$ primitive. As $N$ is odd, $a_0\equiv \a \pmod2$. 
We also take $k_0$ even, and hence $a'$ satisfies $a'\equiv \a \pmod2$.
\end{proof}

Using the symplectic orbifold $Z$ of Theorem \ref{thm:27}, we put coefficients $m_i=p^i$, where $p$ is a fixed prime,
over each of the $16$ surfaces $V_i$ of genus $g_i\in \{1,2\}$, $1\leq i\leq 16$. We get the following main result.

 \begin{corollary} \label{cor:20MRT}
 We can choose a Seifert bundle $\pi:M\to Z$ ramified over the $D_i$ so that $M$ is a K-contact $5$-manifold with
 $H_1(M,\ZZ)=0$ and
  $$
   H_2(M,\ZZ)=\ZZ^{15} \oplus (\mathop{\oplus}_{i=1}^{16} \ZZ_{p^i}^{2g_i}).
   $$
   \end{corollary}

\begin{proof}
We choose local invariants $(m_i,j_i)$ for $D_i$ with $b_{16}=1$. As $m=p^{16}=m_{16}$, we have 
$\frac{b_{16}m}{m_{16}}=1$, so (2) of Lemma \ref{lem:20MRT} is satisfied. By using 
Proposition \ref{prop:acabando}, we can assign local invariants at the singular points. By Lemma
 \ref{lem:20MRT}, there is a Seifert bundle $M\to X$ such that $c_1(M)=[\hat\omega]$ and $c_1(M/m)$
 is primitive. By Proposition \ref{prop:21MRT}, $M$ admits a K-contact structure. And 
 by Theorem \ref{thm:16MRT}, we have that $H_1(M,\ZZ)=0$, and the $2$-homology is as stated.
\end{proof}

To complete the proof of Theorem \ref{thm:main}, it remains to see that $M$ can be chosen
spin or non-spin (Section \ref{sec:w2}) and that $M$ is simply connected (Section \ref{sec:pi1}).

\section{The second Stiefel-Whitney class} \label{sec:w2}

We now study the second Stiefel-Whitney class of a quasi-regular Seifert bundle $\pi:M\to X$,
in order to compute the Barden invariant of $M$.
Let $P=\{x_j\}\subset X$ be the set of singular points, let $B_j\subset X$ be a small ball around
$x_j$, and $S_j=\pi^{-1}(x_j)\subset M$, which is an embedded $S^1$. Let $W_j=\pi^{-1}(B_j)$ 
and $L_j=\bd W_j$. Clearly there is an oriented fiber bundle $S^3\to L_j \to S^1$, hence
$L_j$ is diffeomorphic to $S^1\x S^3$. We consider the Mayer-Vietoris sequence associated to $(M-\sqcup S_j, \sqcup W_j)$,
which gives
 $$
  H^1(\sqcup L_j,\ZZ_2) \to H^2(M,\ZZ_2) \to H^2(M-\sqcup S_j,\ZZ_2) \to H^2(\sqcup L_j,\ZZ_2)=0.
 $$
The first map reads in homology as $H_3(\sqcup L_j,\ZZ_2) \to H_3(M,\ZZ_2)$. As $H_3(L_j,\ZZ_2)$ is spanned
by the normal fiber to $L_j\to S^1$, which shrinks in $M$, this is the zero map. 
Therefore
 $$
 H^2(M,\ZZ_2) \cong H^2(M-\sqcup S_j,\ZZ_2),
 $$
and under this isomorphism we have
 $$
 w_2(M)=w_2(M-S),
 $$
where $\pi:M-S \to X-P$ is a Seifert bundle over a smooth orbifold. 
We also have a map
 $$
 \pi^*:H^2(X-P,\ZZ_2) \to H^2(M-S,\ZZ_2) =H^2(M,\ZZ_2).
 $$
The discussion in \cite[Section 2]{MT2} works also for non-compact manifolds. It says that
 \begin{align}
  w_2(M) &=\pi^*w_2(X-P)+\sum_i(m_i-1)[E_i] \label{eqn:w2-1}\\
   & =\pi^*(w_2(X-P)+\sum_i b_i[D_i] + c_1(B)), \label{eqn:w2-2}
  \end{align}
where $E_i=\pi^{-1}(D_i)$. Note also that $\pi^*[D_i]=m_i [E_i]$.

\begin{proposition}\label{prop:w2}
 Let $\pi:M\to X$ be a quasi-regular Seifert fibration with isotropy locus and local invariants
 $\{(D_i,m_i,b_i)\}$, and $H_1(M,\ZZ)=0$. The mod $2$ cohomology of $X$ is $H^2(X,\ZZ_2)=\ZZ_2^{k+1}$ and 
 the mod $2$ cohomology of $M$ is 
  $$
   H^2(M,\ZZ_2)= H^2(M-S,\ZZ_2)=\ZZ_2^k \oplus (\mathop{\oplus}_{m_i \text{ even}} \ZZ_2^{2g_i}).
  $$
The map $\pi^*:H^2(X-P,\ZZ_2)\to H^2(M-S,\ZZ_2)$ has image onto the first summand $\ZZ_2^k \subset H^2(M-S,\ZZ_2)$.
Its kernel is:
 \begin{itemize}
 \item If all $m_i$ are odd, then $\ker\pi^*$ is one-dimensional spanned by $c_1(B)+\sum b_i [D_i]$.
 \item If $c=\{i\, | \, m_i$ even$\}>0$, then $\ker\pi^*$ is $c$-dimensional, and $\ker\pi^*$ is spanned by those $[D_i]$ with $m_i$ even.
 \end{itemize}
\end{proposition}

\begin{proof}
The proof is the same as in \cite[Proposition 13]{MT2}. We only have to note that $H^1(X-P,\ZZ_2)=0$, 
$H^1(X-P,\ZZ_2)=\ZZ_2^{k+1}$ and $H^3(X-P,\ZZ_2)=0$, by the duality between $H^*(X,\ZZ)$ and $H^*(X-P,\ZZ)$. The
argument in \cite[Proposition 13]{MT2} does not use the compactness of $X$, so it works for $X-P$.
\end{proof}

Note also that the Mayer-Vietoris sequence for $(X-P,B)$ gives an exact sequence
 $$
 0 \to H^1(L,\ZZ_2) \to H^2(X,\ZZ_2) \to H^2(X-P,\ZZ_2) \to H^2(L,\ZZ_2) \to 0.
 $$
Here $H^2(X,\ZZ_2)=\ZZ_2^{k+1}$ and $H^2(X-P,\ZZ_2)=\ZZ_2^{k+1}$, where $k+1=b_2(X)$.
Let $c=\#\{ x\in P \,|\, d(x)$ is even$\}$. Then $H^1(L,\ZZ_2)=H^2(L,\ZZ_2)=\ZZ_2^c$.
A first consequence is that $c\leq b_2(X)$.
Note that this argument works not only for $p=2$, but for other primes.

\begin{corollary} Let $p$ be a prime
 and $c(p)=\#\{ x\in P\, | \, d(x) \equiv 0 \pmod p\}$. Then $c(p)\leq b_2(X)$. \hfill $\Box$
 \end{corollary}

Let $x_j\in P$ with $d(x_j)$ even.
The map $H^2(X-P,\ZZ_2) \to H^2(L,\ZZ_2)$ sends $[D]$ to $1 \in H^2(L_j,\ZZ_2)$ if $x_j\in D$ and to $0$ otherwise.
The map $H^1(L_j,\ZZ_2) \to H^2(X,\ZZ_2)$ is equivalent to $H_2(L_j,\ZZ_2)=\ZZ_2 \to H_2(X-P,\ZZ_2)$ and it is
the immersion of an $\RP^2\subset L_j$ linking $x_j$ into $X$.

\begin{proposition}
If $p=2$ then the manifold $M\to Z$ of Corollary \ref{cor:20MRT} 
is spin. If $p>2$ then we can arrange $c_1(B)$ and $b_i$ so that $M$ is spin or non-spin.
\end{proposition}

\begin{proof}
If $p=2$, then Proposition \ref{prop:w2} says that the map $\pi^*:H^2(Z-P,\ZZ_2) \to H^2(M,\ZZ_2)$ is zero. By 
(\ref{eqn:w2-2}), we have $w_2(M)=0$.

If $p>2$ then all $m_i$ are odd, and we use (\ref{eqn:w2-1}) which says $w_2(M)=\pi^*(w_2(Z-P))$. The 
cohomology of $H^2(Z-P,\ZZ_2)$ is generated by $V_1, \ldots, V_{16}$, where only $V_1,V_2$ touch
the points of $P$. By (\ref{eqn:V12}), $V_i^2=\pm 1$ for $i\geq 3$. As
$\S^2 \equiv w_2(Z-P)\cdot\S \pmod2$ for any surface $\S$ not going through the singular points, we have
 \begin{equation}\label{eqn:wZ}
 w_2(Z-P)= a_1 [V_1]+ a_2 [V_2] + \sum_{i=3}^{16} [V_i]\, ,
 \end{equation}
for some numbers $a_1,a_2$ that we do not need to compute.
The kernel of $\pi^*$ is given by Proposition \ref{prop:w2} to be $c_1(B)+\sum b_i[V_i]$.

The Seifert bundle $\pi:M\to Z$ is determined by the Chern class
 $$
 c_1(M)=c_1(B)+\sum \frac{b_i}{m_i} [V_i],
 $$
where $b_i$ are determined by the local invariants, and
$B$ is a line bundle over $Z$. 
By Proposition \ref{prop:w2}, the manifold $M$ is spin or non-spin according to 
whether (\ref{eqn:wZ}) is proportional to $c_1(B)+ \sum b_i [V_i]$. Taking $\a=\sum b_i[V_i] + a_1[V_1]+a_2[V_2]+\sum_{i\geq 3}[V_i]$,
we can use Lemma \ref{lem:20MRT}(3) to get $c_1(B)\not\equiv \alpha \pmod 2$, and hence $M$ is non-spin.
We can use Lemma \ref{lem:20MRT}(4) to get $c_1(B)\equiv \alpha \pmod 2$, and hence $M$ is spin.
 \end{proof}

\section{The orbifold fundamental group} \label{sec:pi1}

Our last objective is to show that for $p>2$, the $5$-manifold of Corollary \ref{cor:20MRT}  is simply connected, thereby
completing the proof of Theorem \ref{thm:main}. As 
As this $5$-manifold $M$ is a Seifert bundle $M\to Z$, the fundamental group $\pi_1(M)$ is an extension
 $$
 \ZZ \to \pi_1(M)\to \pi_1^{\orb}(Z) \to 0  ,
 $$
where $\pi_1^{\orb}(Z)$ is the \emph{orbifold fundamental group} of $Z$, with the orbifold 
structure given in Theorem \ref{thm:main2}.
This is defined as 
 $$
 \pi_1^{\orb}(Z)=\pi_1(Z-\cup \, V_i)/\la \gamma_i^{m_i}\ra,
 $$
where $\gamma_i$ is a loop around $V_i$ and $m_i=p^i$ is its multiplicity.

By Theorem \ref{thm:25}, the total space of $Z=Z'\#2\overline{\CP}{}^2$ is simply connected. 
We start by looking at $Z'=Y \#_{T_1=C} W$. Let 
 $$
  W^o= W - C\cup A_1\cup A_2\cup \left(\mathop{\cup}\limits_{i=3}^8 E_i\right).
  $$
By Proposition \ref{prop:24}, 
the loops $a,b \in \pi_1(C)$ can
be contracted inside $W^o$. As $W$ is simply connected, the group $\pi_1(W^o)$ is 
generated by loops around $C,A_1,A_2, E_i$, and order $2$ loops around the singular points.
Let also
 $$ 
 Y^o=Y-T_1\cup (S_1\cup T_{p_1'}\cup T_{p_1''}) \cup (S_2\cup T_{p_2'}) \cup  \left(\mathop{\cup}\limits_{i=3}^8T_{p_i}\right).
 $$
This is the quotient 
 $$
 Y^o= \left( (\S_2 -\{p_1',p_1'',p_2',p_i  | 1\leq i\leq 8\})\times (T^2-\{q_1\})\right)/\ZZ_2\, .
 $$
Therefore $\pi_1(Y^o)$ is generated by the loops around $S_1,S_2,T_{p_1'}, T_{p_1''}, T_{p_2'},
T_{p_i}$, $3\leq i\leq 8$, the loops $a_1,b_1,a_2,b_2$ generating $\pi_1(\S_2)$, two loops $\a,\b$ 
generating $\pi_1(T^2)$ and a loop $\ell$ of order $2$ giving the covering. This loop can be a loop around a singular
point (note that all of them are conjugated). 

Take tubular neighbouhoods $U(T_1)$ and $U(C)$ of $T_1$ and $C$, respectively, and perform the gluing 
 $$
 Z'^{o}= \tilde W^o \cup_{\bd U(T_1)=\bd U(C)} Y^o = Z' -  \left(\mathop{\cup}\limits_{i=1}^8 V_i\right).
 $$
Then $\pi_1(Z'^o)$ is generated by loops around each of the surfaces $V_i$
of (\ref{eqn:V12-def}), the loops around the double
points and the loops $\a,\b$. Note that the loops $a_j,b_j$ can be pushed to $a,b\in \pi_1(C)$ and then contracted in
$W^o$.
 
 \begin{lemma} \label{lem:44}
 In $\pi_1(Z'^{o})$, we have $\g_i=\g_j$ for $i,j\geq 3$, $\g_2=\g_i^3$ and $\g_1=\g_i^{-5}$. Moreover $\g_i^8=1$.
 \end{lemma}
 
\begin{proof} 
First consider $V_i$, $3\leq i\leq 8$, and a small neighbourhood $U(V_i)$ around it. The boundary $\bd U(V_i)$ is
a circle bundle over a torus $V_i=E_i\#T_{p_i}$ with Chern class $-1$. Therefore $\pi_1(\bd U(V_i))$ is generated
by loops $\alpha_i,\beta_i$ from the base, and $\gamma_i$ from the fiber, with the relation
$[\alpha_i,\beta_i]=\gamma_i^{-1}$. Note that $\gamma_i$ is the loop around the $V_i$.
Now moving from one fiber to another, we have that $\alpha_i=\alpha$ and
$\beta_i=\beta$ in the complement of all the surfaces.
Therefore $\gamma_i=\gamma_j$, for all $3\leq i,j\leq 8$.
Denote $\Upsilon=\gamma_i$.

Now we look at 
 $$
  V_2= A_2\#  (S_2\cup T_{p_2'})= A_2' \# S_2, 
  $$
where $A_2'=A_2 \# T_{p_2'}$ is a torus with an orbifold point of order $2$ (what we shall glue with an order $2$ point
of $S_2$).
The orbifold fundamental group of $A_2'$ is generated by loops $\alpha_2,\beta_2$ (which
define the same homotopy classes as $\alpha,\beta$ above) and a loop $v_2$ around the point of order $2$, and
the relation $[\alpha_2,\beta_2]v_2=1$.
The orbifold fundamental of $S_2$ is generated by loops $x_2,y_2,z_2,u_2$ around the points of order $2$, where $x_2y_2z_2u_2=1$.
When performing the orbifold connected sum $V_2=A_2'\#  S_2$, we identify $v_2=u_2$, so we have 
 $$
 \pi_1^{\orb}(V_2)=\la \alpha_2,\beta_2,x_2,y_2,z_2 \,|\, x_2^2=y_2^2=z_2^2=[\alpha_2,\beta_2]x_2y_2z_2=1\ra.
 $$
 
Next take a neighbourhood $U(V_2)$ and its boundary $\bd U(V_2)$, and let $\gamma_2$ be the loop in the fiber. 
There is an exact sequence for the Seifert circle bundle
 $$
  0 \to \pi_1(S^1)=\ZZ \to \pi_1(\bd U(V_2)) \to \pi_1^{\orb}(V_2) \to 0.
 $$
Take lifts $\alpha_2,\beta_2,x_2,y_2,z_2 \in \pi_1(\bd U(V_2))$
of the corresponding elements in $\pi_1^{\orb}(V_2)$. They satisfy the equation
 \begin{equation}\label{eqn:relation}
 [\alpha_2,\beta_2]x_2y_2z_2=\gamma_2^{2}\, , \quad x_2^2=y_2^2=z_2^2=\gamma_2,
 \end{equation}
and $\gamma_2$ is central in $\pi_1(\bd U(V_2))$.
We see the exponent of $\gamma_2$ in (\ref{eqn:relation}) as follows. The loops $x_2,y_2,z_2$ are lifted by taking
the boundary of a disc in $\{z_2=\epsilon z_1\}$ in a local chart $(z_1,z_2)\in \CC^2/\ZZ_2$. 
Therefore they go around the fiber $1/2$ turn. So $[\alpha_2,\beta_2]x_2y_2z_2$
has Chern number $3/2$. As $V_2^2= - \frac12$ from (\ref{eqn:V12}), we need
to add $-2$ turns around the fiber to close the loop.

Now we move the loops $\alpha,\beta$ vertically from the general fiber $T_{p}$ to the special 
fiber $S_2$, recalling that $T_{p}\to S_2=T_{p_2}/\ZZ_2$
is a ramified double cover of the $2$-torus over the pillowcase. 
Under this covering $\alpha$ goes to $x_2y_2$ and $\beta$ goes to $z_2x_2$,
so 
 $$
  [\alpha_2,\beta_2]= \alpha_2\beta_2\alpha_2^{-1}\beta_2^{-1}=x_2y_2z_2x_2y_2x_2x_2z_2\gamma_2^{-4}
   =(x_2y_2z_2)^2 \gamma_2^{-3},
  $$
using that $\alpha_2^{-1}=(x_2y_2)^{-1}=y_2^{-1}x_2^{-1}=y_2\g_2^{-1}x_2\g_2^{-1}=y_2x_2\g_2^{-2}$ 
and $\beta_2^{-1}=x_2z_2\g_2^{-2}$.
If we write  
 $$
  \Theta_2=x_2y_2z_2,
  $$
then $[\alpha,\beta]=\Theta_2^2 \gamma_2^{-3}$ and (\ref{eqn:relation}) says that $\Theta_2^3 \g_2^{-3}=\g_2^{2}$, that is
$\Theta_2^3=\g_2^{5}$.
The loops $\alpha,\beta$ can be moved to $\alpha_i,\beta_i$ above, so 
 $$
 \Upsilon=\gamma_i=  \Theta_2^2\g_2^{-3} \, ,
  $$
and raising to the $3$rd power, $\Upsilon^3=\Theta_2^6\g_2^{-9}=\g_2$. 

We work analogously with 
 $$
  V_1=A_1'\#  S_1, \quad A_1'= A_1\# (T_{p_1'}\cup T_{p_1''}),
  $$ 
where $A_1'$ is a genus $2$ curve with a singular point of order $2$. The fundamental group $\pi_1(\bd U(V_1))$ is
generated by loops $\alpha'_1,\beta'_1,\alpha_1'',\beta_1'', x_1,y_1,z_1, \g_1$,
with the relations 
 \begin{equation}\label{eqn:relation2}
  [\alpha'_1,\beta'_1][\alpha_1'',\beta_1''] x_1y_1z_1=\g_1, \qquad x_1^2=y_1^2=z_1^2=\gamma_1,
  \end{equation}
 where $\alpha_1'=\alpha_1''=\alpha$ and $\beta_1'=\beta_1''=\beta$. Writing $\Theta_1=x_1y_1z_1$, we end up 
 with 
$\Theta_1^5 \g_1^{-6}=\g_1$, that is $\Theta_1^5=\g_1^{7}$. Moving the loops 
$\alpha_1',\beta_1'$ to $\alpha,\beta$, we get 
 $$
 \Upsilon=\gamma_i=[\alpha'_1,\beta'_1] =  \Theta_1^2 \g_1^{-3}
  $$
and raising to the $5$th power, $\Upsilon^{5}=\Theta_1^{10}\g_1^{-15}=\g_1^{-1}$. So $\g_1=\Upsilon^{-5}$. 

Finally take a generic section $\S_q$  of the projection $\pi'$, which is a genus $2$ surface intersecting transversally all 
$V_i$, $3\leq i\leq 8$,  in two points, $V_1$ in $5$ points, and $V_2$ in $3$ points. 
This surface $\S_q$ has fundamental group generated by four loops $a_1,b_1,a_2,b_2$ that 
are pushed to $a,b\in \pi_1(C)$ and contracted in $W^o$. 
Therefore, using $\S_q$, we get the relation 
 $$
 \Upsilon^8  \gamma_1^5 \gamma_2^3 =1 \implies \Upsilon^{-8}=1 .
 $$
\end{proof}

Now we move to $Z=Z'\# 2 \overline{\CP}{}^2$ and to 
 $$
 Z^o =Z- \left(\mathop{\cup}\limits_{i=1}^{16} V_i \right).
 $$
For this, we take the extra surfaces $U_1,U_2,U_3,U_4 \subset Z'$, and blow-up at the intersections points.
All the homotopies performed in Lemma \ref{lem:44} can be done without touching $U_1\cup U_2$
and $U_3\cup U_4$. Therefore the result of Lemma \ref{lem:44} still holds in $Z^o$.
After blowing up, we have surfaces $V_{11}=U_1,V_{12}=U_2,V_{13}=U_3,V_{14}=U_4$ and 
two genus $2$ surfaces $V_{15},V_{16}$ The first four surfaces are tori, each of them has a $S^1$-factor
which is either $a_1$ or $b_1$, and the other $S^1$-factor is either $\alpha$ or $\beta$. The first one
can be contracted. The second one can be pushed vertically to any fiber $T_p$. Note also that
$V_i^2=\pm 1$, for $9\leq i\leq 14$. Therefore the loop around $V_i$ can be written as commutators
of the base, and by the contraction of $a_1,b_1$, they are trivial. 

\begin{proposition} \label{prop:acab}
 For $p$ odd, and isotropy coefficients $m_i=p^i$, we have that $\pi_1^{\orb}(Z)$ is a quotient of $\ZZ_2\x\ZZ_2$.
 \end{proposition}
 
 \begin{proof}
 By Lemma \ref{lem:44}, we have that $\g_i^8=1$. Impossing $\g_i^{m_i}=1$ in $\pi_1^{\orb}(Z)$, we 
 get that $\g_i=1$ in $\pi_1^{\orb}(Z)$. By Lemma \ref{lem:44} again, this implies that $\g_1=\g_2=1$.
 Now $[\alpha,\beta]=\g_i=1$ and (\ref{eqn:relation}) and (\ref{eqn:relation2}) implies
  that $x_1y_1z_1=x_1^2=y_1^2=z_1^2=1$ and 
 $x_2y_2z_2=x_2^2=y_2^2=z_2^2=1$. That is $z_j=x_jy_j$ and $(x_jy_j)^2=x_j^2=y_j^2=1$. So
 the group $\la x_j,y_j\ra\cong \ZZ_2\x\ZZ_2$.
  
 Finally, take the special fibers $T_{q_i}$, $i=2,3,4$, to get that $x_1=x_2$, $y_1=y_2$ and $z_1=z_2$.
 Therefore  $\pi_1^{\orb}(X)$
is generated by $x_1,y_1$. In particular it is a quotient of $\ZZ_2\x \ZZ_2$. 
\end{proof}

\begin{corollary}
For $p$ odd, the $5$-manifold $M$ of Corollary \ref{cor:20MRT} is simply connected.
\end{corollary}

\begin{proof} 
First, let us see that $\pi_1(M)$ is abelian.
As $M$ is a  Seifert bundle $M\to Z$, we have an extension 
$\ZZ \to \pi_1(M)\to \pi_1^{\orb}(Z) \to 0$.
If the first map is zero, $\pi_1(M)$ is automatically abelian. If it is not, then we have an extension
$$
 0 \to \ZZ \to \pi_1(M)\to \pi_1^{\orb}(Z) \to 0.   
 $$
 By Proposition \ref{prop:acab}, $\pi_1^{\orb}(Z)$ is a quotient of $\ZZ_2\x\ZZ_2$.
If $\pi_1^{\orb}(Z)= 0$ then $\pi_1(M)=\ZZ$ is abelian. 
If $\pi_1^{\orb}(Z)=\ZZ_2$, then $\pi_1(M)$ is abelian, because the class $\lambda$ generated by $\ZZ$ is
the class of a generic fiber of the Seifert fibration, which is central. Finally, if 
$\pi_1^{\orb}(Z)=\ZZ_2\x\ZZ_2$, then take $x,y\in \pi_1(M)$ going to the generators of $\ZZ_2\x \ZZ_2$
in the quotient. Then  $xy=\lambda^k yx$ for some $k\in \ZZ$. Noting that $\lambda$ is central, we have
$xyx=\lambda^k y x x$ and $xyx=\lambda^{-k} xxy$. But $x^2,y^2\in \la \lambda\ra$ are central,
so $\lambda^kyx^2=\lambda^{-k}x^2 y \implies k=0$. So $x,y$ commute and $\pi_1(M)$ is abelian.
 
In Corollary \ref{cor:20MRT}, we have chosen the Seifert bundle to have $H_1(M,\ZZ)=0$, 
hence $\pi_1(M)=0$ and $M$ is simply connected.
\end{proof}


\begin{thebibliography}{33}


\bibitem{B} \textsc{D. Barden}, {\it Simply connected five-manifolds}, Ann. Math. 82 (1965) 365-385.


\bibitem{BBFMT} \textsc{G. Bazzoni, I. Biswas, M. Fernández, V. Mu\~noz, A. Tralle},
{\it Homotopic properties of Kähler orbifolds},
In: Special Metrics and Group Actions in Geometry. Springer INdAM Series 23. Springer,  23-57. 

\bibitem{BFMT} \textsc{I. Biswas, M. Fern\'andez, V. Mu\~noz, A. Tralle},  
{\it On formality of orbifolds and Sasakian manifolds}, J. Topology 9 (2016) 161-180. 


\bibitem{BB} \textsc{J. Borzellino, V. Brunsden}, {\it On the notions of suborbifold and orbifold embedding},
Algebraic \& Geom. Topology 15 (2015) 2789-2803.

\bibitem{BG} \textsc{C. Boyer, K. Galicki}, {\it Sasakian Geometry}, Oxford Univ. Press, 2007.


\bibitem{CMRV} \textsc{A. Ca\~nas, V. Mu\~noz, J. Rojo, A. Viruel}, {\it A K-contact simply connected $5$-manifold without 
Sasakian structure}, arxiv:1911.08901


\bibitem{CNY} \textsc{B. Cappelletti-Montano, A. de Nicola, I. Yudin}, {\it Hard Lefschetz theorem for Sasakian manifolds}, 
J. Diff. Geom. 101 (2015) 47-66.
    
\bibitem{CNMY} \textsc{B. Cappelletti-Montano, A. de Nicola, J.C. Marrero, I. Yudin}, {\it Examples of
compact K-contact manifolds with no Sasakian metric}, 
Internat. Jour. Geom. Methods in Modern Physics 11 (2014) 1460028.
 
 
\bibitem{Gompf} \textsc{R. Gompf}, {\it A new construction of symplectic manifolds}, Ann.\ Math. (2) 142 (1995) 537-696.

\bibitem{HT} \textsc{B. Hajduk, A. Tralle}, {\it On simply connected compact K-contact non-Sasakian manifolds}, 
J. Fixed Point Theory Appl. 16 (2014) 229-241.

\bibitem{K} \textsc{J. Koll\'ar}, {\it Circle actions on simply connected $5$-manifolds}, Topology, 45 (2006) 643-672.

\bibitem{MR} \textsc{V. Mu\~noz, J.A. Rojo}, {\it Symplectic resolution of orbifolds with homogeneous isotropy},
Geometriae Dedicata 204 (2020) 339-363.

\bibitem{MRT} \textsc{V. Mu\~noz, J.A. Rojo, A. Tralle}, {\it Homology Smale-Barden manifolds with K-contact and Sasakian structures},
Internat. Math. Res. Notices, doi.org/10.1093/imrn/rny205

\bibitem{MT}\textsc{V. Mu\~noz, A. Tralle}, {\it Simply connected K-contact and Sasakian manifolds of dimension 7}, Math. Z. 281 (2015) 457-470

\bibitem{MT2}\textsc{V. Mu\~noz, A. Tralle}, {\it On the classification of Smale-Barden manifolds with Sasakian structures}, arXiv:2002.00457 

\bibitem{OT} \textsc{J. Oprea, A. Tralle}, Symplectic Manifolds with no Kaehler structure, 1997, Springer.

\bibitem{R} \textsc{P. Rukimbira}, {\it Chern-Hamilton conjecture and K-contactness}, {Houston J. Math.} {21} (1995) 709-718.


\bibitem{S} \textsc{S. Smale}, {\it On the structure of 5-manifolds}, Ann. Math. 75 (1962) 38-46.

\bibitem{Wei} \textsc{M. Weilandt}, {\it Suborbifolds, quotients and transversality}, Topology and its Appl. 222 (2017) 293-306.


\end{thebibliography}
\end{document}